\documentclass[twoside, a4paper, 12pt]{amsart}
\usepackage{fullpage}
\usepackage{amsfonts}
\usepackage{amssymb}
\usepackage{enumitem}
\theoremstyle{plain}
\newtheorem*{claim*}{Claim}
\newtheorem{thm}{Theorem}[section]
\newtheorem{corollary}[thm]{Corollary}
\newtheorem{lemma}[thm]{Lemma}
\newtheorem{prop}[thm]{Proposition}
\theoremstyle{definition}

\newtheorem{ex}[thm]{Example}

\newtheorem{remark}[thm]{Remark}
\newtheorem{con}[thm]{Construction}

\newtheorem{prob}[thm]{Open Problem}

\begin{document}
\title{\large{Semigroups whose right ideals are finitely generated}}
\author{Craig Miller}
\address{Department of Mathematics, University of York, UK, YO10 5DD}
\email{craig.miller@york.ac.uk}

\maketitle

\begin{abstract}
We call a semigroup $S$ {\em weakly right noetherian} if every right ideal of $S$ is finitely generated; equivalently, $S$ satisfies the ascending chain condition on right ideals.
We provide an equivalent formulation of the property of being weakly right noetherian in terms of principal right ideals, and we also characterise weakly right noetherian monoids in terms of their acts.\par
We investigate the behaviour of the property of being weakly right noetherian under quotients, subsemigroups and various semigroup-theoretic constructions.  In particular, we find necessary and sufficient conditions for the direct product of two semigroups to be weakly right noetherian.\par
We characterise weakly right noetherian regular semigroups in terms of their idempotents.  We also find necessary and sufficient conditions for a strong semilattice of completely simple semigroups to be weakly right noetherian.  
Finally, we prove that a commutative semigroup $S$ with finitely many archimedean components is weakly (right) noetherian if and only if $S/\mathcal{H}$ is finitely generated.
\end{abstract}

\section{\large{Introduction}\nopunct}

A {\em finiteness condition} for a class of universal algebras is a property that is satisfied by at least all finite members of that class.  Some of the most important finiteness conditions are {\em ascending chain conditions}.  The study of ascending chain conditions on ideals of rings, initiated by Noether in the early part of the last century, has been instrumental in the development of the structure theory of rings.  A ring is {\em right} (resp.\ {\em left}) {\em Noetherian} if it satisfies the ascending chain condition on right (resp.\ left) ideals,
and {\em Noetherian} if it is both right Noetherian and left Noetherian.
Noetherian rings play a key role in many major ring-theoretic results, such as Hilbert's basis theorem and Krull's intersection theorem.\par
We call a semigroup $S$ {\em weakly right noetherian} if every right ideal of $S$ is finitely generated.\footnote{Weakly right noetherian semigroups are also known in the literature as {\em right noetherian}.  However, we use the term `right noetherian' to denote semigroups whose right congruences are all finitely generated.}
Similarly, a semigroup is {\em weakly left noetherian} if every left ideal is finitely generated.  We call a semigroup {\em weakly noetherian} if it both weakly right noetherian and weakly left noetherian.  It is clear that each of these properties is a finiteness condition.  In this paper we will focus on weakly {\em right} noetherian semigroups.
Such semigroups have received a significant amount of attention; see for instance \cite{Aubert, Davvaz, Jespers, Sat}.\par
Related to the property of being weakly right noetherian is the stronger condition that every right congruence is finitely generated; we call semigroups satisfying this condition {\em right noetherian}.  Such semigroups were studied systematically in \cite{Miller1}, and had previously been considered in \cite{Hotzel, Kozhukhov1, Kozhukhov2}.  Another related notion is that of the universal right congruence being finitely generated, which was first considered in \cite{Dandan}.  The stronger condition that every right congruence of finite index is finitely generated (where index means the number of classes) was introduced and studied in \cite{Miller2}.\par
This paper is structured as follows.  In Section \ref{sec:preliminaries} we provide the necessary preliminary material.  
In Section \ref{sec:fundamentals} we present some equivalent formulations of the property of being weakly right noetherian.
In Sections \ref{sec:quotients} and \ref{sec:sub} we explore how a semigroup and its substructures and quotients relate to one another with regard to the property of being weakly right noetherian.  
We then investigate how this property behaves under various semigroup-theoretic constructions in Section \ref{sec:con}.  Specifically, we consider direct products, free products, semilattices of semigroups, Rees matrix semigroups, Brandt extensions and Bruck-Reilly extensions.
Section \ref{sec:reg} is concerned with regular semigroups.  We first consider regular semigroups in general, and then focus on the important subclasses of inverse semigroups and completely regular semigroups.
Finally, in Section \ref{sec:comm}, we consider commutative semigroups.  The main result of that section is a necessary and sufficient condition for a commutative semigroup with finitely many archimedean components to be weakly noetherian.

\section{\large{Preliminaries}\nopunct}
\label{sec:preliminaries}

In this section we establish some basic definitions and notation.  We begin by providing some set-theoretic definitions.\par  
A relation $\leq$ on a set $P$ is said to be {\em preorder} if it is both reflexive and transitive.  If a preorder is also symmetric, then it is an {\em equivalence relation}.  On the other hand, if a preorder is antisymmetric then it is a {\em partial order}.\par
A {\em poset} is a set $P$ together with a partial order $\leq$.  Given any set $X,$ a collection $P$ of subsets of $X$ forms a poset under the partial order of inclusion.  In particular, the set of all right ideals of a semigroup is a poset (under $\subseteq$).\par
Let $(P, \leq_P)$ and $(Q, \leq_Q)$ be two posets.  A map $\theta : P\to Q$ is said to be {\em order-preserving} if $x\leq_P y$ implies $x\theta\leq_Q y\theta$ for all $x, y\in P.$  A map $\theta : P\to Q$ is an {\em isomorphism} if both $\theta$ and $\theta^{-1}$ are order-preserving (i.e.\ $x\leq_P y$ if and only if $x\theta\leq_Q y\theta$ for all $x, y\in P$) and $\theta$ is a bijection.  We say that $P$ and $Q$ are {\em isomorphic} if there exists an isomorphism between them.  Note that to show that a map $\theta : P\to Q$ is an isomorphism, it suffices to prove that $\theta$ is {\em surjective} and that $x\leq_P y$ if and only if $x\theta\leq_Q y\theta$ for all $x, y\in P.$\par
Two elements $a$ and $b$ of a poset $P$ are said to be {\em comparable} if either $a\leq b$ or $b\leq a$; otherwise, $a$ and $b$ are {\em incomparable}.  A subset of $P$ in which any two elements are comparable is called a {\em chain}.  An {\em antichain} of $P$ is a subset consisting of pairwise incomparable elements.\\
~\par
We now turn our attention to semigroups.  We refer the reader to \cite{Howie} for a more comprehensive introduction to semigroup theory.  Throughout the remainder of the section, $S$ will denote a semigroup.\par
We denote by $S^1$ the monoid obtained from $S$ by adjoining an identity if necessary (if $S$ is already a monoid, then $S^1=S$).  Similarly, we denote by $S^0$ the semigroup with zero obtained from $S$ by adjoining a zero if necessary.\par
Let $M$ be a monoid with identity $1.$  An element $a\in M$ is said to be {\em right invertible} if there exists $b\in M$ such that $ab=1.$  {\em Left invertible} elements are defined dually.  An element of $M$ is called a {\em unit} if it is both right invertible and left invertible.  The units of $M$ form a group, called the {\em group of units} of $M,$ which we denote by $U(M).$\par
We denote the set of idempotents of $S$ by $E(S).$  If $S=E(S),$ it is called a {\em band}.  A {\em semilattice} is a commutative band.  The multiplication in a semilattice $E$ induces the following partial order: $$e\geq f\iff ef=f.$$
In this way we may view $E$ as a meet-semilattice in the order-theoretic sense.  Conversely, any order-theoretic meet-semilattice may be viewed as a commutative band with meet taken as the binary operation.\par
If $a, b\in S$ are such that $a=aba$ and $b=bab,$ then $b$ is called an {\em inverse} of $a.$  The semigroup $S$ is said to be {\em regular} if every element of $S$ has an inverse.  If, additionally, the inverse of each element of $S$ is unique, then $S$ is an {\em inverse semigroup}.  It is well known that a semigroup is inverse if and only if it is regular and its idempotents form a semilattice \cite[Theorem 5.1.1]{Howie}.\par
A subset $I\subseteq S$ is said to be a {\em right ideal} of $S$ if $IS\subseteq I.$  Left ideals are defined dually, and an {\em ideal} of $S$ is a subset that it is both a right ideal and a left ideal.\par
Given a subset $X\subseteq S,$ the {\em right ideal generated by} $X$ is the set $XS^1.$
A right ideal $I$ of $S$ is said to be {\em finitely generated} if it can be generated by a finite set.\par
Note that a right ideal of $S$ can be generated by a set {\em as a right ideal} or {\em as a semigroup}.  For proper right ideals, we will always use the term `generate' in the former sense.  When we say that $S$ is generated by a set $X,$ we mean `generated as a semigroup', unless stated otherwise, and we write $S=\langle X\rangle.$  We note that a right ideal can be finitely generated as a right ideal but not as a semigroup; e.g.\ any non-finitely group $G$ is certainly finitely generated as a right ideal.\par  
A {\em right congruence} on $S$ is an equivalence relation $\rho$ on $S$ such that $(a, b)\in\rho$ implies $(ac, bc)\in\rho$ for all $a, b, c\in S$; {\em left congruences} are defined analagously.  A {\em congruence} is a relation that is both a right congruence and left congruence.  For a congruence $\rho$ on $S,$ we denote the congruence class of an element $a\in S$ by $[a]_{\rho}.$\par
Recall that a semigroup is {\em right noetherian} if every right congruence is finitely generated.  (A right congruence $\rho$ on $S$ is {\em finitely generated} if there exists a finite set $X\subseteq\rho$ such that $\rho$ is the smallest right congruence on $S$ containing $X.$)  Right noetherian semigroups are weakly right noetherian \cite[Lemma 2.7]{Miller1}, but the converse certainly does not hold.  Indeed, unlike the situation for rings, the lattice of right congruences on a semigroup is not in general isomorphic to the lattice of right ideals.  For example, groups have no proper right ideals, but the lattice of right congruences on a group is isomorphic to its lattice of subgroups.  Consequently, groups are trivially weakly right noetherian, but a group is right noetherian if and only if all its subgroups are finitely generated \cite[Proposition 2.14]{Miller1}.\par
The most essential tools for understanding the structure of a semigroup are its Green's relations $\mathcal{L}, \mathcal{R}, \mathcal{H}, \mathcal{D}$ and $\mathcal{J}.$
They are defined as follows.\par 
Two elements $a, b\in S$ are $\mathcal{L}$-{\em related} if they generate the same principal left ideal, i.e. $S^1a=S^1b.$
Similarly, two elements $a, b\in S$ are $\mathcal{R}$-{\em related} if they generate the same principal right ideal.
Green's relation $\mathcal{H}$ is defined as $\mathcal{H}=\mathcal{L}\cap\mathcal{R}.$
Two elements $a, b\in S$ are $\mathcal{D}$-{\em related} if there exists $c\in S$ such that $a\,\mathcal{L}\,c$ and $c\,\mathcal{R}\,b.$
Finally, if two elements $a, b\in S$ generate the same principal ideal (i.e. $S^1aS^1=S^1bS^1$), then they are said to be $\mathcal{J}$-{\em related}.\par
It is obvious from the definitions that $\mathcal{L}, \mathcal{R}, \mathcal{H}$ and $\mathcal{J}$ are equivalence relations on $S,$ and it turns out that $\mathcal{D}$ is also an equivalence relation.
Moreover, Green's relation $\mathcal{L}$ is a right congruence on $S$ and $\mathcal{R}$ is a left congruence on $S.$\par
Green's relation $\mathcal{R}$ defines a preorder $\leq_{\mathcal{R}}$ on $S,$ given by
$$a\leq_{\mathcal{R}}b\iff aS^1\subseteq bS^1.$$
The preorder $\mathcal{R}$ induces a partial order on the set of $\mathcal{R}$-classes of $S$: $R_a\leq R_b$ if and only if $a\leq_{\mathcal{R}}b.$  It is easy to see that the poset of $\mathcal{R}$-classes of $S$ is isomorphic to the poset of principal right ideals of $S,$ via the isomorphism $R_a\mapsto aS^1.$  Similarly, one can define preorders $\leq_{\mathcal{L}}$ and $\leq_{\mathcal{J}},$ leading to partial orders on the sets of $\mathcal{L}$-classes and $\mathcal{J}$-classes, respectively.\par
Note that when we need to distinguish between Green's relations on different semigroups, we will write them with the semigroup as a subscipt; i.e.\ $\mathcal{K}_S$ stands for $\mathcal{K},$ where $\mathcal{K}$ is any of Green's relations on $S.$\par
It is easy to see that the following inclusions between Green's relations hold: 
$$\mathcal{H}\subseteq\mathcal{L},\, \mathcal{H}\subseteq\mathcal{R},\, \mathcal{L}\subseteq\mathcal{D},\, \mathcal{R}\subseteq\mathcal{D},\, \mathcal{D}\subseteq\mathcal{J}.$$
It can be easily shown that every right (resp.\ left) ideal is a union of $\mathcal{R}$-classes (resp.\ $\mathcal{L}$-classes), and every ideal is a union of $\mathcal{J}$-classes.
A semigroup with no proper (right) ideals is called ({\em right}) {\em simple}.  A simple semigroup has a single $\mathcal{J}$-class; if it is right simple, then it has a single $\mathcal{R}$-class.  If $S$ has a zero $0,$ $S^2\neq\{0\},$ and $\{0\}$ is the only proper ideal of $S,$ then it is called {\em $0$-simple}.\par
Given an ideal $I$ of $S,$ the {\em Rees quotient} of $S$ by $I,$ denoted by $S/I,$  is the set $(S\!\setminus\!I)\cup\{0\}$ with multiplication given by
$$a\cdot b=\begin{cases}
ab&\text{ if }a, b, ab\in S\!\setminus\!I,\\
0&\text{ otherwise.}
\end{cases}$$
\par
Let $J$ be a $\mathcal{J}$-class of $S.$  The {\em principal factor} of $J$ is defined as follows.  If $J$ is the unique minimal ideal of $S,$ called the {\em kernel} of $S,$ its principal factor is itself.  Otherwise, the principal factor of $J$ is the Rees quotient of the subsemigroup $S^1xS^1,$ where $x$ is any element of $J,$ by the ideal $(S^1xS^1)\!\setminus\!J.$\par
The {\em principal factors} of $S$ are the principal factors of its $\mathcal{J}$-classes.  The kernel of $S,$ if it exists, is simple; all other principal factors are either $0$-simple or {\em null} (every product of two elements equals zero).\par
A {\em principal series} of a semigroup $S$ is a finite chain of ideals $$K(S)=I_1\subset I_2\subset\dots\subset I_n=S,$$
where $K(S)$ is the kernel of $S,$ and $I_k$ is maximal in $I_{k+1}$ for each $i\in\{1, \dots, n-1\}.$  The kernel $K(S)$ and the Rees quotients $I_{k+1}/I_k$ are the principal factors of $S.$\par
It is folklore that a semigroup has a principal series of length $n$ if and only if it has exactly $n$ $\mathcal{J}$-classes.\\
~\par
Closely related to the notion of one-sided ideals is that of {\em semigroup acts}.  We provide some basic definitions about acts; one should consult \cite{Kilp} for more information.\par
A ({\em right}) {\em $S$-act} is a non-empty set $A$ together with a map 
$$A\times S\to A, (a, s) \mapsto as$$
such that $a(st)=(as)t$ for all $a\in A$ and $s, t\in S.$  (If $S$ is a monoid, we also require that $a1=a$ for all $a\in A$.)
For instance, $S$ itself is an $S$-act via right multiplication.\par
A subset $B$ of an $S$-act $A$ is a {\em subact} of $A$ if $bs\in B$ for all $b\in B$ and $s\in S.$
Note that the right ideals of $S$ are precisely the subacts of the $S$-act $S.$\par
Given an $S$-act $A$ and a subact $B$ of $A,$ the {\em Rees quotient} of $A$ by $B,$ denoted by $A/B,$ is the set $(A\!\setminus\!B)\cup\{0\}$ with action given by
$$a\cdot s=
\begin{cases}
as &\text{if }as\in A\!\setminus\!B\\
0 &\text{otherwise,}
\end{cases}$$
and $0\cdot s=0,$ for all $a\in A\!\setminus\!B$ and $s\in S.$  It can be easily verified that $A/B$ is an $S$-act via the above action.\par
A subset $U$ of an $S$-act $A$ is a {\em generating set} for $A$ if $A=US^1,$
and $A$ is said to be {\em finitely generated} if it has a finite generating set.\par
We call an $S$-act $A$ {\em noetherian} if every subact of $A$ is finitely generated; equivalently, $A$ satisfies the ascending chain condition on its subacts.  In particular, the $S$-act $S$ being noetherian is equivalent to $S$ being a weakly right noetherian semigroup.

\section{\large{Equivalent Formulations and Elementary Facts}\nopunct}
\label{sec:fundamentals}

We begin this section by presenting equivalent characterisations of weakly right noetherian semigroups in terms of the ascending chain condition and maximal condition on right ideals.  The proof of this result is essentially the same as that of the analogue for rings and is omitted.

\begin{prop}
\label{acc}
The following are equivalent for a semigroup $S$:
\begin{enumerate}
\item $S$ is weakly right noetherian;
\item $S$ satisfies the ascending chain condition on right ideals; that is,
every ascending chain $I_1\subseteq I_2\subseteq\cdots$ of right ideals of $S$ eventually terminates;
\item every non-empty set of right ideals of $S$ has a maximal element.
\end{enumerate}
\end{prop}

We now provide characterisations of weakly right noetherian semigroups in terms of their principal right ideals and also in terms of their $\mathcal{R}$-class structure.

\begin{thm}
\label{principal}
The following are equivalent for a semigroup $S$:
\begin{enumerate}
\item $S$ is weakly right noetherian;
\item $S$ satisfies the ascending chain condition on principal right ideals and contains no infinite antichain of principal right ideals (under $\subseteq$).
\item the poset of $\mathcal{R}$-classes of $S$ contains no infinite strictly ascending chain or infinite antichain.
\end{enumerate}
\end{thm}

\begin{proof}
$(1)\Rightarrow(2).$  By Proposition \ref{acc}, $S$ certainly satisfies the ascending chain condition on principal right ideals.  The fact that $S$ contains no infinite antichain of principal right ideal was proven in \cite[Lemma 1.6]{Hotzel}, but we provide a proof for completeness.\par
Suppose for a contradiction that there exists an infinite antichain $\{a_iS^1 : i\in\mathbb{N}\}$ of principal right ideals of $S.$  For each $n\in\mathbb{N},$ let $I_n=\{a_1, \dots, a_n\}S^1.$  Suppose that $I_m=I_n$ for some $m\leq n.$  Then $a_n\in a_iS^1$ for some $i\leq m.$  It must be the case that $i=m=n,$ for otherwise the incomparability of $a_iS^1$ and $a_nS^1$ would be contradicted. 
But then we have an infinite strictly ascending chain $$I_1\subsetneq I_2\subsetneq\dots$$ of right ideals of $S,$ contradicting Proposition \ref{acc}.\par
$(2)\Rightarrow(1).$  Suppose that $S$ is not weakly right noetherian yet the poset of principal right ideals of $S$ {\em does} satisfy the ascending chain condition.  We need to construct an infinite antichain of principal right ideals of $S.$\par
By Proposition \ref{acc} there exists an infinite strictly ascending chain
$$I_1\subsetneq I_2\subsetneq\dots$$
of right ideals of $S.$  Choose elements $a_1\in I_1$ and $a_k\in I_k\!\setminus\!I_{k-1}$ for $k\geq 2.$  Then certainly $a_kS^1$ is not contained in any $a_jS^1, j<k,$ since $a_jS^1\subseteq I_j$ and $a_k\in I_k\!\setminus\!I_{j}.$\par 
Consider the infinite set $\{a_iS^1 : i\in\mathbb{N}\}$ of principal right ideals of $S.$  This set contains a maximal element, say $a_{k_1}S^1$; that is, $a_{k_1}S^1$ is not contained in any $a_jS^1, j\neq k_1.$  Indeed, if this were not the case, then there would exist an infinite strictly ascending chain of principal right ideals of $S,$ contradicting the assumption.\par
Now consider the infinite set $\{a_iS^1 : i\geq k_1+1\}.$  Again, this set contains a maximal element, say $a_{k_2}S^1.$  Thus $a_{k_2}S^1$ is not contained in $a_jS^1$ for any $j>k_1, j\neq k_2.$  In fact, $a_{k_2}S^1$ is not contained in $a_jS^1$ {\em for any} $j\in\mathbb{N}\!\setminus\!\{k_2\},$ since, as observed above, $a_{k_2}S^1$ is not contained in any $a_jS^1, j<k.$\par
Continuing this process ad infinitum, we obtain an infinite antichain $\{a_{k_i}S^1 : i\in\mathbb{N}\}$ of principal right ideals of $S,$ as required.\par
$(2)\Leftrightarrow(3).$  This follows from the fact, established in Section \ref{sec:preliminaries}, that the poset of $\mathcal{R}$-classes of $S$ is isomorphic to the poset of principal right ideals of $S.$
\end{proof}

\begin{corollary}
\label{R-classes}
Any semigroup with finitely many $\mathcal{R}$-classes is weakly right noetherian.  In particular, all finite semigroups and all right simple semigroups (which include groups) are weakly right noetherian.
\end{corollary}

\begin{remark}
The condition that every (two-sided) ideal of a semigroup $S$ is finitely generated has been considered in \cite{Aubert, Keh}.  By an argument essentially the same as the proof of Theorem \ref{principal}, this condition is equivalent to $S$ satisfying the ascending chain condition on principal ideals and containing no infinite antichain of principal ideals, and also to $S$ containing no infinite strictly ascending chain or infinite antichain of $\mathcal{J}$-classes.  Any weakly right noetherian semigroup satisfies this condition (since every ideal is a one-sided ideal), but the converse does not hold.  Indeed, any simple semigroup trivially satisfies the condition that every ideal is finitely generated, but there exist simple semigroups that are not right noetherian; e.g.\ any completely simple semigroup with infinitely many $\mathcal{R}$-classes (see Corollary \ref{completely simple} below).
\end{remark}

The following result shows that there exist semigroups that satisfy the ascending chain condition on principal right ideals but are not weakly right noetherian.

\begin{prop}
\label{free}
Let $X$ be a non-empty set.  Then the free semigroup $F_X$ on $X$ satisfies the ascending chain condition on principal right ideals.  However, $F_X$ is weakly right noetherian if and only if $|X|=1.$
\end{prop}

\begin{proof}
Consider two elements $u, v\in F_X.$  Clearly $uF_X^1\subsetneq vF_X^1$ if and only if $v$ is a proper prefix of $u,$ in which case $|u|>|v|.$  It follows that there cannot exist an infinite strictly ascending chain of principal right ideals of $F_X.$\par
If $|X|=1,$ then clearly $F_X\cong\mathbb{N}$ contains no incomparable elements, so it is weakly right noetherian by Theorem \ref{principal}.\par
Suppose $|X|\geq2,$ and choose distinct elements $x, y\in X.$  For $i\neq j,$ the element $x^iy$ is not a prefix of $x^jy,$ so $F_X$ contains an infinite antichain $\{(x^iy)F_X^1 : i\in\mathbb{N}\}$ of principal right ideals.  Hence, $F_X$ is not weakly right noetherian by Theorem \ref{principal}.
\end{proof}

\begin{remark}
We can readily deduce from Proposition \ref{free} that the property of being weakly right noetherian is not closed under subsemigroups.  Indeed, the free semigroup $F_X$ is a subsemigroup of the free group on $X,$ which is certainly weakly right noetherian.
\end{remark}

Monoid acts play the analogous role in the theory of monoids to that of modules in the theory of rings.  It is well known that a ring $R$ is right Noetherian if and only if every finitely generated right $R$-module is {\em Noetherian} (i.e.\ it satisfies the ascending chain condition on its submodules) \cite[Corollary 1.4]{Goodearl}.
We now present the analogue of this result for monoid acts.

\begin{prop}
\label{acts}
The following are equivalent for a monoid $M$:
\begin{enumerate}
 \item $M$ is weakly right noetherian;
 \item every finitely generated right $M$-act is noetherian.
\end{enumerate}
\end{prop}

\begin{proof}
$(1)\Rightarrow(2)$.  Let $A$ be a right $M$-act with a finite generating set $X,$ and let $B$ be a subact of $A.$
For each $x\in X,$ we define a set $$I_x=\{m\in M : xm\in B\}.$$
Let $X^{\prime}$ be the set of elements in $X$ such that $I_x\neq\emptyset.$
Since $B$ is a subact of $A,$ we have that $I_x$ is a right ideal of $S$ for each $x\in X^{\prime}.$
Since $M$ is weakly right noetherian, for each $x\in X^{\prime}$ there exists a finite set $U_x\subseteq I_x$ such that $I_x=U_xM.$
We claim that $B$ is generated by the set $$U=\bigcup_{x\in X^{\prime}}xU_x.$$
Indeed, let $b\in B.$  Since $b\in A,$ we have that $b=xm$ for some $x\in X$ and $m\in M.$
Now $m\in I_x,$ so $m=un$ for some $u\in U_x$ and $n\in M,$ and hence $b=(xu)n\in UM.$\par
$(2)\Rightarrow(1)$.  This follows from the fact that the right ideals of $M$ are subacts of the cyclic right $M$-act $M.$
\end{proof}

For commutative semigroups, clearly the properties of being weakly right noetherian and being weakly left noetherian coincide.  It is a well-known result, due to R{\'e}dei \cite{Redei}, that every congruence on a finitely generated commutative semigroup is finitely generated; that is:

\begin{thm}\cite{Redei}
\label{fg comm}
Every finitely generated commutative semigroup is noetherian (and hence weakly noetherian).
\end{thm}

In the remainder of this section we state some useful facts about weakly right noetherian semigroups.
The following lemma is well known and will be used repeatedly throughout the remainder of the paper, usually without explicit mention.

\begin{lemma}
\label{finite gen set}
Let $S$ be a semigroup and let $X$ be a subset of $S.$  If the right ideal $XS^1$ of $S$ is finitely generated, then there exists a finite subset $Y\subseteq X$ such that $XS^1=YS^1.$
\end{lemma}

Let $S$ be a semigroup.  An element $s\in S$ is said to be {\em decomposable} if $s\in S^2.$  An element is {\em indecomposable} if it is not decomposable.

\begin{lemma}
\label{indecomposable}
Let $S$ be a weakly right noetherian semigroup.  Then $S$ has only finitely many indecomposable elements.
\end{lemma}

\begin{proof}
Since $S$ is weakly right noetherian, $S$ is finitely generated as a right ideal; 
that is, there exists a finite set $X\subseteq S$ such that $S=XS^1.$
Therefore, we have that $S\!\setminus\!X\subseteq S^2,$ so $S$ has at most $|X|$ indecomposable elements.
\end{proof}

We now show that a semigroup composed of a finite union of weakly right noetherian subsemigroups is also weakly right noetherian.

\begin{lemma}
\label{finite union}
Let $S$ be a semigroup, and suppose that $S$ is a union of subsemigroups $S_1, \dots, S_n.$
If each $S_i$ is weakly right noetherian, then $S$ is also weakly right noetherian.
\end{lemma}

\begin{proof}
Let $I$ be a right ideal of $S.$  For $j\in\{1, \dots, n\}$, let $I_j$ be the restriction of $I$ to $S_j.$  Then $I_j$ is a right ideal of $S_j.$
Since $S_j$ is weakly right noetherian, $I_j$ is generated by some finite set $X_j.$
We claim that $I$ is generated by the finite set $X=\bigcup_{i=1}^nX_i.$
Indeed, if $a\in I,$ then $a\in I_j=X_jS_j^1$ for some $j\in\{1, \dots, n\}.$
\end{proof}

\section{\large{Quotients and Ideals}\nopunct}
\label{sec:quotients}

In this section we consider the relationship between a semigroup and its quotients and ideals with regard to the property of being weakly right noetherian.
We first show that this property is closed under quotients (or, equivalently, homomorphic images).

\begin{lemma}
\label{quotient}
Let $S$ be a semigroup and let $\rho$ be a congruence on $S.$  If $S$ is weakly right noetherian, then so is $S/\rho.$
\end{lemma}

\begin{proof}
Let $I$ be a right ideal of $S/\rho,$ and define a set $$J=\{a\in S : [a]_{\rho}\in I\}.$$
It is clear that $J$ is a right ideal of $S.$
Since $S$ is weakly right noetherian, $J$ is generated by a finite set $X.$
We claim that $\rho$ is generated by the finite set $Y=\{[x]_{\rho} : x\in X\}.$
Indeed, let $u\in I.$  Select $a\in S$ such that $[a]_{\rho}=u.$  Then $a=xs$ for some $x\in X$ and $s\in S^1.$  If $s=1,$ then $u=[x]_{\rho}\in Y.$  Otherwise, we have $u=[x]_{\rho}[s]_{\rho}\in Y(S/\rho),$ as required.
\end{proof}

\begin{remark}
The converse of Lemma \ref{quotient} does not hold.  Indeed, the free semigroup $F_X$ with $|X|\geq2$ is not weakly right noetherian by Proposition \ref{free}, but there certainly exist quotients of $F_X$ that are weakly right noetherian.
\end{remark}

If $\rho$ is a congruence contained in $\mathcal{R},$ then the converse of Lemma \ref{quotient} holds.  In fact, we have:

\begin{lemma}
\label{congruence in R}
Let $S$ be a semigroup and let $\rho\subseteq\mathcal{R}$ be a congruence on $S.$  Then the poset of $\mathcal{R}$-classes of $S$ is isomorphic to the poset of $\mathcal{R}$-classes of $S/\rho.$  In particular, $S$ is weakly right noetherian if and only if $S/\rho$ is weakly right noetherian.
\end{lemma}

\begin{proof}
Let $P$ be the poset of principal right ideals of $S,$ and let $Q$ be the poset of principal right ideals of $S/\rho.$  Since the poset of principal right ideals of a semigroup is isomorphic to the poset of $\mathcal{R}$-classes, it suffices to prove that $P$ and $Q$ are isomorphic.  It then follows from Theorem \ref{principal} that $S$ is weakly right noetherian if and only if $S/\rho$ is weakly right noetherian.\par
Define a map $\theta : P\to Q, aS^1\mapsto [a]_{\rho}(S/\rho)^1.$
Clearly $\theta$ is surjective, so we just need to show that $aS^1\subseteq bS^1$ if and only if $[a]_{\rho}(S/\rho)^1\leq [b]_{\rho}(S/\rho)^1.$  It is clear that the forward direction holds.  Conversely, if $[a]_{\rho}(S/\rho)^1\subseteq[b]_{\rho}(S/\rho)^1$ then $[a]_{\rho}=[b]_{\rho}u$ for some $u\in(S/\rho)^1.$  If $u=1,$ then $(a, b)\in\rho\subseteq\mathcal{R},$ so $aS^1=bS^1.$  If $u\in S/\rho,$ then $u=[s]_{\rho}$ for some $s\in S.$  Then $(a, bs)\in\rho\subseteq\mathcal{R},$ so $aS^1=(bs)S^1\subseteq bS^1,$ as required.
\end{proof}

The following question naturally arises from Lemma \ref{congruence in R}.

\begin{prob}
Given a semigroup $S$ and congruence $\rho$ on $S,$ what is the relationship between the poset of $\mathcal{R}$-classes of $S$ and the poset of $\mathcal{R}$-classes of $S/\rho$?
\end{prob}

\begin{remark}
It is natural to wonder whether certain results about weakly right noetherian semigroups, such as Theorem \ref{principal} and Lemma \ref{quotient}, can be derived from more general results about posets.  In particular, inspired by Theorem \ref{principal}, one could investigate the finiteness condition that a poset contains no infinite ascending chain or infinite antichain.  For instance, is this condition closed under homomorphic images?
\end{remark}

The next result states that if a right ideal $I$ of a semigroup $S$ is weakly right noetherian and the Rees quotient of the $S$-act $S$ by $I$ (where $I$ is regarded as a subact of $S$) is noetherian, then $S$ is weakly right noetherian.

\begin{prop}
\label{right ideal and quotient}
Let $S$ be a semigroup and let $I$ be a right ideal of $S.$  If $I$ is weakly right noetherian and $S/I$ is noetherian (as an $S$-act), then $S$ is weakly right noetherian.
\end{prop}

\begin{proof}
Let $J$ be a right ideal of $S.$  Suppose first that $I\cap J=\emptyset.$  Then $J\subseteq S\!\setminus\!I$ and $J$ may be viewed as a subact of $S/I.$  Since $S/I$ is noetherian, there exists a finite set $U\subseteq I$ such that $J=US^1.$  Thus $J$ is finitely generated as a right ideal of $S.$\par
Now suppose that $I\cap J\neq\emptyset.$  Then $I\cap J$ is a right ideal of $I.$  Since $I$ is weakly right noetherian, $I\cap J$ is generated by some finite set $X.$
Considering $J$ as an $S$-act, the Rees quotient $A=J/(I\cap J)$ is a subact of $S/I.$  Since $S/I$ is noetherian, there exists a finite set $Y\subseteq A$ such that $A=YS^1.$  Let $Y^{\prime}=Y\!\setminus\!\{0\}.$
We claim that $J=(X\cup Y^{\prime})S^1.$\par
Indeed, let $a\in J.$  If $a\in I,$ then $a\in X(I\cap J)^1.$
If $a\in J\!\setminus\!I,$ then $a\in A,$ so $a=ys$ for some $y\in Y$ and $s\in S^1.$  Since $a\in S\!\setminus\!I,$ we must have $y\in Y^{\prime}.$  In either case, we have $a\in(X\cup Y^{\prime})S^1,$ as required.
\end{proof}

\begin{corollary}
\label{ideal extension}
Let $S$ be a semigroup and let $I$ be an ideal of $S$.
If both $I$ and the Rees quotient $S/I$ are weakly right noetherian, then $S$ is weakly right noetherian.
\end{corollary}

\begin{proof}  
We shall prove that $S/I,$ viewed as the Rees quotient of the $S$-act $S$ by the subact $I,$ is a noetherian $S$-act.  It then follows from Proposition \ref{right ideal and quotient} that $S$ is weakly right noetherian.\par
So, let $A$ be a subact of $S/I.$  We need to prove that $A$ is finitely generated.  Fix an element $z\in I.$  We claim that $A$ is a right ideal of $S/I$ (considered as a semigroup).  Indeed, let $a\in A$ and $u\in S/I.$  If $a, u\in S\!\setminus\!I,$ then $au\in A$ since $A$ is a subact of $S/I.$  Otherwise, we have $a=0$ or $u=0,$ in which case $au=0\in A.$  (Since $A$ is a subact of $S/I,$ we have $0=a\cdot z\in A.$)
Since $S$ is weakly right noetherian, there exists a finite set $X\subseteq A$ such that $A=X(S/I)^1.$  We claim that $A=XS^1.$  Indeed, if $a\in A,$ then $a=xu$ for some $x\in X$ and $u\in(S/I)^1.$  If $a\in S^1\!\setminus\!I,$ then $u\in S^1\!\setminus\!I.$  Otherwise, $a=0=xz.$  This completes the proof.
\end{proof}

Recall that in a semigroup $S$ with a principal series $K(S)=I_1\subseteq\dots \subseteq I_n=S,$ the kernel $K(S)$ and the Rees quotients $I_{k+1}/I_k$ are the principal factors of $S.$  Therefore, if all the principal factors are weakly right noetherian, then by successively applying Corollary \ref{ideal extension}, we deduce that $S$ is weakly right noetherian.

\begin{corollary}
\label{principal series}
Let $S$ be a semigroup with a principal series.  If all the principal factors of $S$ are weakly right noetherian, then $S$ is weakly right noetherian.
\end{corollary}

Although ideals of weakly right noetherian semigroup are not in general weakly right noetherian (see, for instance, Example \ref{free comm ex}); they do satisfy the ascending chain condition on {\em principal} right ideals.  In fact, we prove a stronger statement:

\begin{prop}
\label{acc ideal}
Let $S$ be a semigroup and let $I$ be an ideal of $S.$  If $S$ satisfies the ascending chain condition on principal right ideals, then so does $I.$
\end{prop}

\begin{proof}
Consider an infinite ascending chain $$a_1I^1\subseteq a_2I^1\subseteq\cdots$$ of principal right ideals of $I.$  Then for each $i, j\in\mathbb{N}$ with $i<j,$ there exists $u_{i,j}\in I^1$ such that $a_i=a_{j}u_{i,j}.$
Clearly we have an infinite ascending chain 
$$a_1S^1\subseteq a_2S^1\subseteq\cdots$$
of principal right ideals of $S.$  By assumption, there exists $n\in\mathbb{N}$ such that $a_mS^1=a_nS^1$ for all $m\geq n.$  Let $m>n.$  If $u_{n,m}=1,$ then $a_n=a_m.$  Now suppose that $u_{n,m}\in I.$  There exists $s_m\in S^1$ such that $a_m=a_ns_m,$ and hence 
$$a_m=a_ns_m=a_mu_{n,m}s_m=a_n(s_mu_{n,m}s_m).$$
We have that $s_mu_{n,m}s_m\in I$ since $I$ is an ideal, so $a_m\in a_nI.$  Since $a_nI^1\subseteq a_mI^1,$ we conclude that $a_nI^1=a_mI^1.$  Since $m$ was chosen arbitrarily, we have shown that the above ascending chain of principal right ideals of $I$ terminates at $a_nI^1.$
\end{proof}

\section{\large{Subsemigroups}\nopunct}
\label{sec:sub}

As mentioned previously, subsemigroups of weakly right noetherian semigroups need not be weakly right noetherian themselves.  In this section we explore various situations in which the property of being weakly right noetherian passes from a semigroup $S$ to a subsemigroup $T,$ and vice versa.

\begin{lemma}
\label{R, subsemigroup}
Let $S$ be a semigroup, and let $T$ be a subsemigroup of $S$ such that $S\!\setminus\!T$ is contained in a finite union of $\mathcal{R}$-classes.
If $T$ is weakly right noetherian, then $S$ is weakly right noetherian.
\end{lemma}

\begin{proof}
Let $I$ be a right ideal of $S.$
Then $I\cap T$ is a right ideal of $T.$
Since $T$ is weakly right noetherian, there exists a finite set $X\subseteq I\cap T$ such that $I\cap T=XT^1.$\par
Let $R_1, \dots, R_n$ be the $\mathcal{R}$-classes that intersect with $S\!\setminus\!T.$  For each $i\in\{1, \dots\, n\},$ fix $r_i\in R_i.$
It is easy to see that if $I$ intersects an $\mathcal{R}$-class $R,$ then $R\subseteq I.$
We claim that $I$ is generated by the finite set $$Y=X\cup\{r_i : R_i\subseteq I, 1\leq i\leq n\}.$$ 
Indeed, let $a\in I.$  If $a\in T,$ then $a\in I\cap T=XT^1.$
If $a\in S\!\setminus\!T,$ then $a=r_is$ for some $i\in\{1, \dots\, n\}$ and $s\in S^1.$
Hence, in either case, we have that $a\in YS^1.$
\end{proof}

\begin{remark}
The converse of Lemma \ref{R, subsemigroup} does not hold.  Indeed, in Remark \ref{pf not wrn} an example is provided of a weakly right noetherian semigroup $S$ with a subsemigroup $T$ such that $S\!\setminus\!T$ is a group and $T$ is not weakly right noetherian.
\end{remark}

In \cite{Wallace} Wallace introduced the idea of Greens' relations taken relative to a subsemigroup.  Let $S$ be a semigroup and let $T$ be a subsemigroup of $S.$  The {\em $T$-relative Green's relation} $\mathcal{R}^T$ on $S$ is given by $$a\,\mathcal{R}^T\,b\iff aT^1=bT^1.$$
The relation $\mathcal{L}^T$ is defined dually, and $\mathcal{H}^T=\mathcal{R}^T\cap\mathcal{L}^T.$  The relations $\mathcal{D}^T$ and $\mathcal{J}^T$ are defined in a similar way.
All of these relations are equivalence relations on $S,$ and they respect $T$ in the sense that each class lies entirely in $T$ or entirely in $S\!\setminus\!T.$
The subsemigroup $T$ is said to have {\em finite Green index} in $S$ if there are only finitely many $\mathcal{H}^T$-classes in $S\!\setminus\!T.$  The notion of Green index for subsemigroups was introduced in \cite{Gray}.

\begin{prop}
Let $S$ be a semigroup, and let $T$ be a subsemigroup of $S$ such that $S\!\setminus\!T$ is a finite union of $\mathcal{R}^T$-classes.
Then $S$ is weakly right noetherian if and only if $T$ is weakly right noetherian.
\end{prop}

\begin{proof}
($\Rightarrow$) Let $I$ be a right ideal of $T$.
Let $I^{\prime}=IS^1;$ that is, the right ideal of $S$ generated by $I.$
Since $S$ is weakly right noetherian, we have that $I^{\prime}$ is generated by some finite subset $X\subseteq I$ by Lemma \ref{finite gen set}.
Let $R_1, \dots, R_n$ be the $\mathcal{R}^T$-classes in $S\!\setminus\!T,$  and for each $i\in\{1, \dots\, n\}$ fix $r_i\in R_i.$
We claim that $I$ is generated by the finite set $$Y=X\cup\{xr_i\in I : x\in X, 1\leq i\leq n\}.$$
Indeed, let $a\in I.$  Then $a\in I^{\prime},$ so $a=xs$ for some $x\in X$ and $s\in S^1.$
If $s\in T^1,$ then $a\in YT^1.$
If $s\in S\!\setminus\!T,$ then $s\in R_i$ for some $i\in\{1, \dots, n\},$ so there exist $t, u\in T^1$ such that $s=r_it$ and $r_i=su.$
Then $xr_i=(xs)u=au\in I,$ and hence $a=xs=(xr_i)t\in YT^1,$ as required.\par
($\Leftarrow$) Each $\mathcal{R}^S$-class is a union of $\mathcal{R}^T$-classes by \cite[Proposition 9]{Gray}, so $S\!\setminus\!T$ is contained in a finite union of $\mathcal{R}^S$-classes, and hence $S$ is weakly right noetherian by Lemma \ref{R, subsemigroup}.
\end{proof}

\begin{corollary}
\label{Green}
Let $S$ be a semigroup and let $T$ be subsemigroup of $S$ with finite Green index.  Then $S$ is weakly right noetherian if and only if $T$ is weakly right noetherian.
\end{corollary}

\begin{corollary}
\label{large}
Let $S$ be a semigroup, and let $T$ be a subsemigroup of $S$ such that $S\!\setminus\!T$ is finite.
Then $S$ is weakly right noetherian if and only if $T$ is weakly right noetherian.\par
In particular, $S$ is weakly right noetherian if and only if $S^1$ is weakly right noetherian if and only if $S^0$ is weakly right noetherian.
\end{corollary}

Let $S$ be a semigroup and $T$ a subsemigroup of $S.$  It is easy to see that Green's $\mathcal{R}$-preorder on $T$ is contained in the restriction of Green's $\mathcal{R}$-preorder on $S$ to $T$; that is, $$\leq_{\mathcal{R}_T}~\subseteq~\leq_{\mathcal{R}_S}\cap~(T\times T).$$
We say that $T$ {\em preserves $\mathcal{R}$} (in $S$), or is {\em $\mathcal{R}$-preserving}, if $$\leq_{\mathcal{R}_T}~=~\leq_{\mathcal{R}_S}\cap~(T\times T).$$
It can be easily shown that if $T$ preserves $\mathcal{R},$ then ${\mathcal{R}_T}$ is the restriction of $\mathcal{R}_S$ to $T.$
The next result states that the property of being weakly right noetherian is inherited by $\mathcal{R}$-preserving subsemigroups.

\begin{prop}
\label{R preorder}
Let $S$ be a semigroup and let $T$ be an $\mathcal{R}$-preserving subsemigroup of $S.$  If $S$ is weakly right noetherian, then so is $T.$
\end{prop}

\begin{proof}
Let $I$ be a right ideal of $T,$ and let $I^{\prime}=IS^1.$
Since $S$ is weakly right noetherian, $I^{\prime}=XS^1$ for some finite subset $X\subseteq I.$
If $a\in I,$ then $a\in xS^1$ for some $x\in X,$ so $a\leq_{\mathcal{R}_S}\!x.$  By assumption, we have $a\leq_{\mathcal{R}_T}\!x,$ so $a\in xT^1.$  Thus $I=XT^1$ is finitely generated.
\end{proof}

For a regular subsemigroup $T$ of a semigroup $S,$ Green's relation $\mathcal{R}_T$ is the restriction of $\mathcal{R}_S$ to $T$ (likewise, $\mathcal{L}_T$ and $\mathcal{H}_T$ are the restrictions to $T$ of $\mathcal{L}_S$ and $\mathcal{H}_S,$ respectively) \cite[Proposition 2.4.2]{Howie}.  In fact, $T$ preserves $\mathcal{R}.$  Indeed, if $a, b\in T$ and $a\leq_{\mathcal{R}_S}b,$ then there exists $s\in S^1$ such that $a=bs.$  Letting $b^{\prime}$ be any inverse of $b,$ we have that $a=bb^{\prime}bs=b(b^{\prime}a)\in bT,$ so $a\leq_{\mathcal{R}_T}b,$ as required.  Thus, by Proposition \ref{R preorder} we have:

\begin{corollary}
\label{regular subsemigroup}
Let $S$ be a semigroup with a regular subsemigroup $T.$  If $S$ is weakly right noetherian, then so is $T.$
\end{corollary}

We say that a semigroup $S$ has {\em local right identities} if $a\in aS$ for every $a\in S.$
It is easy to show that the class of semigroups with local right identities includes all regular semigroups, right simple semigroups and monoids.
The notion of having local right identities will be crucial in Section \ref{subsec:dp}.

\begin{corollary}
\label{right ideal, lri}
Let $S$ be a semigroup, and let $I$ be a right ideal of $S$ with local right identities.  If $S$ is weakly right noetherian, then so is $I.$
\end{corollary}

\begin{proof}
We prove that $I$ preserves $\mathcal{R}$ in $S.$  We just need to show that $\leq_{\mathcal{R}_S}\cap~(I\times I)~\subseteq~\leq_{\mathcal{R}_I.}$
So, let $(a, b)\in I\times I$ and $a\leq_{\mathcal{R}_S}b.$  Then $a\in bS^1.$  Since $I$ is a right ideal with local right identities, we have $$a\in bS^1\subseteq(bI)S^1=b(IS^1)\subseteq bI,$$ 
so $a\leq_{\mathcal{R}_I}b,$ as required.
\end{proof}

A subsemigroup $T$ of a semigroup $S$ is called {\em right unitary} (in $S$) if it satisfies the following condition: for all $a\in T$ and $b\in S,$ if $ab\in T$ then $b\in T.$\par
Clearly a right unitary subsemigroup is $\mathcal{R}$-preserving, so we deduce the following corollary, first proven in \cite{Jespers}, from Proposition \ref{R preorder}.

\begin{corollary}
\cite[Lemma 1.1(1)]{Jespers}
\label{right unitary}
Let $S$ be a semigroup and let $T$ be a right unitary subsemigroup of $S.$
If $S$ is weakly right noetherian, then so is $T.$
\end{corollary}

If the complement of a subsemigroup is a left ideal, then the subsemigroup is right unitary, so we have:

\begin{corollary}
\label{complement left ideal}
Let $S$ be a semigroup with a subsemigroup $T$ such that $S\!\setminus\!T$ is a left ideal of $S.$
If $S$ is weakly right noetherian, then so is $T.$
\end{corollary}

\begin{corollary}
\label{rnideal}
Let $S$ be a semigroup with a subsemigroup $T$ such that $S\!\setminus\!T$ is a weakly right noetherian ideal of $S.$
Then $S$ is weakly right noetherian if and only if $T$ is weakly right noetherian.
\end{corollary}

\begin{proof}
The direct implication follows from Corollary \ref{complement left ideal}.
For the converse, let $I=S\!\setminus\!T$.  
Since $T$ is weakly right noetherian, so is $S/I\cong T\cup\{0\}$ by Corollary \ref{large}.
It now follows from Corollary \ref{ideal extension} that $S$ is weakly right noetherian.
\end{proof} 

The final part of this section concerns semigroups with a kernel.

\begin{prop}
\label{kernel}
Let $S$ be a semigroup with a minimal right ideal.  If $S$ is weakly right noetherian, then the kernel $K=K(S)$ is a finite union of pairwise incomparable $\mathcal{R}_K$-classes and is hence weakly right noetherian.
\end{prop}

\begin{proof}
The kernel $K$ is the union of all the minimal right ideals of $S$ \cite[Theorem 2.1]{Clifford}.  Due to their minimality, these minimal right ideals are single $\mathcal{R}_S$-classes and are pairwise incomparable.  By Theorem \ref{principal}, there are only finitely many of them.  By \cite[Theorem 2.4]{Clifford}, each of these minimal right ideals is a single $\mathcal{R}_K$-class (so $K$ is $\mathcal{R}$-preserving).  It now follows from Corollary \ref{R-classes} that $K$ is weakly right noetherian.
\end{proof}

Let $S$ be a semigroup with a zero $0.$  The {\em right socle} of $S,$ denoted by $\Sigma_r(S),$ is the union of $\{0\}$ and all the $0$-minimal right ideals of $S.$  It turns out that $\Sigma_r(S)$ is an ideal of $S,$ as noted in \cite[Section 6.3]{C&P}.  A similar argument to the one in the proof of Proposition \ref{kernel} yields:

\begin{prop}
\label{right socle}
Let $S$ be a semigroup with zero.  If $S$ is weakly right noetherian, then the right socle $\Sigma_r(S)$ consists of $\{0\}$ and a finite union of incomparable $\mathcal{R}$-classes, and hence $\Sigma_r(S)$ is weakly right noetherian.
\end{prop}

\begin{prob}
Let $S$ be a weakly right noetherian semigroup with a kernel $K.$  Is $K$ also weakly right noetherian?  If $K=\{0\}$ and $S$ has a $0$-minimal ideal $M,$ is $M$ weakly right noetherian?
\end{prob}

\section{\large{Constructions}\nopunct}
\label{sec:con}

In this section we investigate the behaviour of the property of being weakly right noetherian under the following semigroup-theoretic constructions: direct products, free products, semilattices of semigroups, Rees matrix semigroups, Brandt extensions and Bruck-Reilly extensions.

\subsection{Direct products\nopunct}
\label{subsec:dp}
\ \par
\vspace{0.5em}
The problem of whether the property of being weakly right noetherian is preserved under direct products was previously considered in \cite{Davvaz}, where it was shown that the direct product of two weakly right noetherian commutative monoids is weakly right noetherian \cite[Theorem 3.8]{Davvaz}.\par
The purpose of this subsection is to provide necessary and sufficient conditions for the direct product of two semigroups to be weakly right notherian.\par
Recall that a semigroup $S$ has local right identities if $a\in aS$ for every $a\in S.$

\begin{thm}
\label{direct product}
Let $S$ and $T$ be two semigroups with $S$ infinite.
\begin{enumerate}[leftmargin=*]
\item Suppose $T$ is infinite.  Then $S\times T$ is weakly right noetherian if and only if both $S$ and $T$ are weakly right noetherian and have local right identities.
\item Suppose $T$ is finite.  Then $S\times T$ is weakly right noetherian if and only if $S$ is weakly right noetherian and $T$ has local right identities.
\end{enumerate}
\end{thm}

In order to prove Theorem \ref{direct product}, we first present a couple of preliminary results.

\begin{lemma}
\label{lri factor}
Let $S$ and $T$ be two semigroups with $S$ infinite.
If $S\times T$ is weakly right noetherian, then $T$ has local right identities.
\end{lemma}

\begin{proof}
Let $t\in T,$ and let $I$ be the right ideal of $S\times T$ generated by the set $\{(s, t) : s\in S\}.$
Since $S\times T$ is weakly right noetherian, there exists a finite set $X\subseteq S$ such that $I$ is generated by the set $\{(x, t) : x\in X\}.$\par
Choose $s\in S\!\setminus\!X.$  Then $(s, t)=(x, t)w$ for some $x\in X$ and $w\in(S\times T)^1.$ 
Since $s\neq x,$ we conclude that $w\in S\times T,$ so $w=(u, v)$ for some $u\in S$ and $v\in T.$  It follows that $t=tv\in tT.$
Since $t$ was chosen arbitrarily, $T$ has local right identities.
\end{proof}

\begin{prop}
\label{lri dp}
Let $S$ and $T$ be two semigroups with local right identities.  
If both $S$ and $T$ are weakly right noetherian, then $S\times T$ is weakly right noetherian.
\end{prop}

\begin{proof}
Let $I$ be a right ideal of $S\times T.$
For each $a\in S,$ define a set $$I_a^T=\{t\in T : (a, t)\in I\}.$$
We claim that $I_a^T$ is a right ideal of $T.$  Indeed, let $t\in I_a^T$ and $u\in T.$  Since $S$ has local right identities, there exists $s\in S$ such that $as=a.$
Since $I$ is a right ideal of $S\times T,$ we have that $(a, tu)=(a, t)(s, u)\in I,$ so $tu\in I_a^T.$\par
Similarly, for each $b\in T$ we define a right ideal $$I_b^S=\{s\in S: (s, b)\in I\}$$ of $S.$
We now make the following claim.\par
{\em 
\begin{enumerate}[leftmargin=*]
 \item There exists a finite set $X\subseteq S$ with the following property:
for each $a\in S,$ there exists $x\in X$ such that $a\in xS$ and $I_a^T=I_x^T.$
 \item There exists a finite set $Y\subseteq T$ with the following property:
for each $b\in T,$ there exists $y\in Y$ such that $b\in yT$ and $I_b^S=I_y^S.$
\end{enumerate}
}
\begin{proof}[Proof of claim.]
Clearly it is enough to prove (1).  We shall just write $I_a$ for $I_a^T.$
Suppose there are infinitely many right ideals of the form $I_a.$  Note that for any $a, s\in S,$ we have $I_a\subseteq I_{as}.$
Write $S=J_1.$  Since $S$ is weakly right noetherian, there exists a finite set $X_1\subseteq J_1$ such that $J_1=X_1S^1.$
In fact, we have $J_1=X_1S,$ since $S$ having local right identities implies that $X_1\subseteq X_1S.$
By our assumption, there exists $x_1\in X_1$ such that there are infinitely many $a\in x_1S$ with $I_a\neq I_{x_1}.$  Consider the set $$J_2=\{a\in x_1S : I_a\neq I_{x_1}\}.$$
If $a\in J_2$ and $s\in S,$ then $I_{x_1}\subsetneq I_a\subseteq I_{as},$ so $as\in J_2,$ and hence $J_2$ is a right ideal of $S.$
Since $S$ is weakly right noetherian, there exists a finite set $X_2\subseteq J_2$ such that $J_2=X_2S,$ 
and there exists $x_2\in X_2$ such that there are infinitely many $a\in x_2S$ with $I_a\neq I_{x_2}.$
Continuing in this way, we obtain an infinite ascending chain $$I_{x_1}\subset I_{x_2}\subset\cdots$$ of right ideals of $T,$
but this contradicts the fact that $T$ is weakly right noetherian. 
Hence, there exists a finite set $U\subseteq S$ such that $I_a\in\{I_u : u\in U\}$ for every $a\in S.$
For each $u\in U,$ let $K_u$ be the right ideal of $S$ generated by the set 
$$H_u=\{a\in S : I_a=I_u\}.$$
Since $S$ is weakly right noetherian, there exists finite set $X_u\subseteq H_u$ such that $H_u=X_uS.$
Now set $X=\bigcup_{u\in U}X_u$.  It is clear that $X$ satisfies the condition in the statement of the claim.
\end{proof}
Returning to the proof of Proposition \ref{lri dp}, we claim that $I$ is generated by the finite set $Z=I\cap(X\times Y).$
Indeed, let $(a, b)\in I.$  Then $a\in I_b^S$ and $b\in I_a^T.$
By the above claim, there exist $x\in X$ and $s\in S$ such that $a=xs$ and $I_a^T=I_x^T,$
and there exist $y\in Y$ and $t\in T$ such that $b=yt$ and $I_b^S=I_y^S.$
We have that
$$a\in I_b^S=I_y^S\Longrightarrow(a, y)\in I\Longrightarrow y\in I_a^T=I_x^T\Longrightarrow(x, y)\in I.$$
We conclude that $$(a, b)=(x, y)(s, t)\in Z(S\times T),$$ as required.
\end{proof}

We are now ready to prove Theorem \ref{direct product}.

\begin{proof}[Proof of Theorem \ref{direct product}.]
Is $S\times T$ is weakly right noetherian, then both $S$ and $T,$ being homomorphic images of $S\times T,$ are weakly right noetherian by Lemma \ref{quotient}, and $T$ has local right identities by Lemma \ref{lri factor}.  If $T$ is infinite, then $S$ also has local right identities by Lemma \ref{lri factor}.\par
For the case that $T$ is infinite, the converse follows immediately from Proposition \ref{lri dp}.  
Now assume that $T$ is finite, and suppose that $S$ is weakly right noetherian and $T$ has local right identities.
We have that $S^1$ is weakly right noetherian by Corollary \ref{large}, and clearly $S^1$ has local right identities.
Therefore, by Proposition \ref{lri dp}, we have that $S^1\times T$ is weakly right noetherian.
Since $(S^1\times T)\!\setminus\!(S\times T)$ is finite, it follows from Corollary \ref{large} that $S\times T$ is weakly right noetherian.
\end{proof}

\subsection{Free products\nopunct}
\ \par
\vspace{0.5em}
We shall define the free product of two semigroups (resp.\ monoids) in terms of semigroup (resp.\ monoid) presentations.  For more information about semigroup and monoid presentations, we refer the reader to \cite[Section 9.1]{Clifford}.\par
Given two semigroups $S$ and $T$ defined by presentations $\langle X\,|\,Q\rangle$ and $\langle Y\,|\,R\rangle,$ respectively, the {\em semigroup free product} of $S$ and $T,$ denoted by $S\ast T,$ is the semigroup defined by the presentation $\langle X, Y\,|\,Q, R\rangle.$
If $S$ and $T$ are monoids, then the {\em monoid free product} of $S$ and $T,$ denoted by $S\ast_1T,$ is the monoid defined by the presentation $\langle X, Y\,|\,Q, R, 1_S=1_T\rangle,$ where $1_S$ is a fixed word over $X$ representing the identity of $S$ and $1_T$ is a fixed word over $Y$ representing the identity of $T.$
In the case that $S$ and $T$ are groups, the monoid free product $S\ast_1T$ coincides with the group free product of $S$ and $T$; this fact is noted in \cite[Section 8.2, p.\ 266]{Howie}.\par
In the following we provide necessary and sufficient conditions for the semigroup (resp.\ monoid) free product of two semigroups (resp.\ monoids) to be weakly right noetherian.

\begin{thm}
\label{semi free product}
Let $S$ and $T$ be two semigroups.  Then $S\ast T$ is weakly right noetherian if and only if both $S$ and $T$ are trivial.
\end{thm}

\begin{proof}
We denote $S\ast T$ by $U.$\par
$(\Rightarrow).$  Suppose that $T$ is non-trivial, and choose $a\in S$ and distinct elements $b, c\in T.$
For $i\in\mathbb{N},$ let $u_i=(ab)^iaca.$  Let $I$ be the right ideal of $U$ generated by the set $X=\{u_i : i\in\mathbb{N}\}.$  For any $i\in\mathbb{N},$ the element $u_i$ cannot we written as $u_jv$ for any $j\neq i$ and $v\in U,$ so $X$ is a minimal generating set for $I$ and hence $I$ is not finitely generated.  Therefore, $U$ is not weakly right noetherian.\par
$(\Leftarrow).$  The semigroup $U$ is defined by the presentation 
$$\langle e, f\,|\,e^2=e, f^2=f\rangle.$$  Then $U$ is the disjoint union of the following subsemigroups: $$\langle ef\rangle\cong\mathbb{N},\, \langle fe\rangle\cong\mathbb{N},\, \{(ef)^ie : i\geq 0\}\cong\mathbb{N}_0,\, \{(fe)^if : i\geq 0\}\cong\mathbb{N}_0.$$  Since $\mathbb{N}$ and $\mathbb{N}_0$ are weakly right noetherian, it follows from Lemma \ref{finite union} that $U$ is weakly right noetherian.
\end{proof}

Before stating our next result, we first make some definitions.\par  
Let $M$ and $N$ be two disjoint monoids.  A {\em reduced sequence} over $M$ and $N$ is a sequence $(u_1, \dots, u_n)$ such that: $u_i\in(M\cup N)\!\setminus\!\{1_M, 1_N\}$ for each $i\in\{1, \dots, n\}$; $(u_i, u_{i+1})\in M\times N$ or $(u_i, u_{i+1})\in N\times M$ for each $i\in\{1, \dots, n-1\}.$\par
Every non-identity element of $M\ast_{1}N$ can be uniquely written as $u_1\dots u_n$ for some reduced sequence $(u_1, \dots, u_n)$ over $M$ and $N$ \cite[Section 9.4]{C&P}; the elements $u_i, 1\leq i\leq n,$ are called the {\em free factors} of $u_1\dots u_n.$

\begin{thm}
Let $M$ and $N$ be two monoids.  Then $M\ast_1N$ is weakly right noetherian if and only if one of the following holds:
\begin{enumerate}
\item $M$ is weakly right noetherian and $N$ is trivial, or vice versa;
\item both $M$ and $N$ contain precisely two elements;
\item both $M$ and $N$ are groups.
\end{enumerate}
\end{thm}

\begin{proof}
We denote $M\ast_1N$ by $U.$  If $N$ is trivial, then $M$ is isomorphic to $U,$ so we may assume that both $M$ and $N$ are non-trivial.\par
$(\Rightarrow).$  Suppose for a contradiction that $|N|\geq 3$ and at least one of $M$ and $N$ is not a group.  A monoid in which every element is right invertible is a group.  Therefore, we can choose $a\in M\!\setminus\!\{1\}$ and distinct elements $b, c\in N\!\setminus\!\{1\}$ such that at least one of $a, b, c$ is not right invertible.
Let $u_i=(ab)^iacab$ for $i\in\mathbb{N},$ and let $I$ be the right ideal of $U$ generated by $\{u_i : i\in\mathbb{N}\}.$  Suppose that $I$ is finitely generated.  Then it can be generated by a finite set $$X=\{u_i : 1\leq i\leq k\}.$$
Then $u_{k+1}=u_iv$ for some $i\in\{1, \dots, k\}$ and $v\in U.$
Since at least one of $a, b, c$ is not right invertible, the first $2i+2$ free factors of $u_iv$ are $a, b, \dots, a, b, a, c.$  But the free factor in position $2i+2$ of $u_{k+1}$ is $b,$ so we have a contradiction.  Hence, $I$ is not finitely generated and $U$ is not weakly right noetherian.\par
$(\Leftarrow).$  If $M$ and $N$ are both groups, then $U$ is also a group and hence weakly right noetherian.\par
Now suppose that both $M$ and $N$ contain precisely two elements and that $N$ is not a group.
Then $N$ is isomorphic to the two-element semilattice $\{1, 0\},$ and $M$ is isomorphic to either $\{1, 0\}$ or $\mathbb{Z}_2.$\par
If $M\cong\{1, 0\},$ then $U$ is isomorphic to $V^1$ where $V$ is the free product of two trivial semigroups.  It follows from Theorem \ref{semi free product} and Corollary \ref{large} that $U$ is weakly right noetherian.\par
If $M\cong\mathbb{Z}_2,$ then $U$ is defined by the monoid presentation 
$$\langle a, b\,|\,a^2=1, b^2=b\rangle.$$
Let $u_1=ab, u_2=ba$ and $u_3=bab,$ and let $U_i=\{u_i^n : n\geq 0\}$ for $i=1, 2, 3.$
Also, let $U_4=\{u_1^n(aba) : n\geq 0\}.$  Then each $U_i$ is isomorphic to the free monogenic monoid $\mathbb{N}_0$, and $U=\bigcup_{i=1}^4U_i.$
Since $\mathbb{N}_0$ is weakly right noetherian, it follows from Lemma \ref{finite union} that $U$ is weakly right noetherian.
\end{proof}

\subsection{Semilattices of semigroups\nopunct}
\ \par
\vspace{0.5em}
Let $Y$ be a semilattice and let $(S_{\alpha})_{\alpha\in Y}$ be a family of disjoint semigroups, indexed by $Y.$
If $S=\bigcup_{\alpha\in Y}S_{\alpha}$ is a semigroup such that $S_{\alpha}S_{\beta}\subseteq S_{\alpha\beta}$ for all $\alpha, \beta\in Y,$
then $S$ is called a {\em semilattice of semigroups}, and we denote it by $S=\mathcal{S}(Y, S_{\alpha}).$\par 
Now let $S=\bigcup_{\alpha\in Y}S_{\alpha},$ and suppose that for each $\alpha, \beta\in Y$ with $\alpha\geq\beta$ there exists a homomorphism $\phi_{\alpha, \beta} : S_{\alpha}\to S_{\beta}.$  Furthermore, assume that:
\begin{itemize}
 \item for each $\alpha\in Y,$ the homomorphism $\phi_{\alpha, \alpha}$ is the identity map on $S_{\alpha}$;
 \item for each $\alpha, \beta, \gamma\in Y$ with $\alpha\geq\beta\geq\gamma$, we have $\phi_{\alpha, \beta}\,\phi_{\beta, \gamma}=\phi_{\alpha, \gamma}.$
\end{itemize}
For $a\in S_{\alpha}$ and $b\in S_{\beta},$ we define
$$ab=(a\phi_{\alpha, \alpha\beta})(b\phi_{\beta, \alpha\beta}).$$
With this multiplication, $S$ is a semilattice of semigroups.
In this case we call $S$ a {\em strong semilattice of semigroups} and denote it by $S=\mathcal{S}(Y, S_{\alpha}, \phi_{\alpha, \beta}).$\par 
In the remainder of this section, we investigate under what conditions a (strong) semilattice of semigroups is weakly right noetherian.\par 
The following characterisation of weakly right noetherian semilattices follows immediately from Theorem \ref{principal}.

\begin{prop}\cite[Proposition 3.1]{Gould}
\label{wn semilattice}
Let $Y$ be a (meet-)semilattice.  Then $Y$ is weakly noetherian if and only if it contains no infinite strictly ascending chain or infinite antichain of elements.
\end{prop}

Since the semilattice $Y$ is a homomorphic image of $S=\mathcal{S}(Y, S_{\alpha}),$ 
by Lemma \ref{quotient} we have:

\begin{lemma}
\label{semilattice}
Let $S=\mathcal{S}(Y, S_{\alpha})$ be a semilattice of semigroups.
If $S$ is weakly right noetherian, then $Y$ is weakly noetherian.
\end{lemma}

For a semilattice of semigroups $\mathcal{S}(Y, S_{\alpha})$ to be weakly right noetherian, it is not required that all the $S_{\alpha}$ be weakly right noetherian.  In order to show this, we consider the following construction, which will be used again later in the paper.

\begin{con}
\label{con}
Let $S$ and $T$ be two semigroups with homomorphisms $\theta, \phi : S\to T.$
Let $N_T=\{x_t : t\in T\}\cup\{0\}$ be a null semigroup disjoint from $S.$
We define a mulitiplication on $S\cup N_T,$ extending those on $S$ and $N_T,$ as follows:
$$s\cdot x_t=x_{(s\theta)t},\:x_t\cdot s=x_{t(s\phi)}.$$
With this multiplication, $S\cup N_T$ is a semigroup.  We denote it by $\mathcal{U}(S, T; \theta, \phi).$  We simplify this expression in the case that $S=T$ by only writing $S$ once, and similarly if $\theta=\phi.$ 
\end{con}

We may view $\mathcal{U}(S, T; \theta, \phi)$ as a semilattice of semigroups, where the structure semilattice is $\{\alpha, 0\}$ and the corresponding subsemigroups are $S$ and $N_T,$ respectively.  Every non-zero element of a null semigroup is indecomposable, so infinite null semigroups are not weakly right noetherian by Lemma \ref{indecomposable}.  Therefore, the following result yields the desired counterexample.

\begin{prop}
\label{con prop}
Let $S$ and $T$ be two semigroups with homomorphisms $\theta, \phi : S\to T$ where $\phi$ is surjective, and let $U=\mathcal{U}(S, T; \theta, \phi).$  Then $U$ is weakly right noetherian if and only if $S$ is weakly right noetherian.
\end{prop}

\begin{proof}
If $U$ is weakly right noetherian, then since $U\!\setminus\!S=N_T$ is an ideal of $U,$ we have that $S$ is weakly right noetherian by Corollary \ref{complement left ideal}.\par 
Conversely, suppose $S$ is weakly right noetherian, and let $I$ be a right ideal of $U.$ 
Now $I\cap S$ is either empty or a right ideal of $S$; in the latter case it is generated by a finite set $Y$ since $S$ is weakly right noetherian.
We have that $T$ is weakly right noetherian by Lemma \ref{quotient}.  In particular, $T=AT^1$ for some finite set $A\subseteq T.$  For each $a\in A,$ define a set $$I_a=\{s\in S : x_as\in I\}.$$
Let $A^{\prime}$ be the set of elements in $A$ such that $I_a\neq\emptyset.$
For each $a\in A^{\prime},$ we have that $I_a$ is a right ideal of $S,$ 
so it is generated by some finite set $U_a$ since $S$ is weakly right noetherian.
We claim that $I$ is generated by the finite set $$Z=Y\cup\biggl(\bigcup_{a\in A^{\prime}}x_aU_a^1\biggr).$$
Let $u\in I.$  If $u\in S,$ then $u\in I\cap S=YS^1.$
Clearly $0\in ZU,$ so we just need to consider the case that $u=x_t$ for some $t\in T.$
Then $t=av$ for some $a\in A$ and $v\in T^1.$  If $v=1,$ then $u=x_a\in Z.$
Otherwise, let $s\in S$ be such that $s\theta=v,$ so $s\in I_a=U_aS^1.$
It follows that $u=x_as\in(x_aU_a)S^1,$ as required.
\end{proof}

\begin{remark}
\label{pf not wrn}
In general, principal factors of weakly right noetherian semigroups need not be weakly right noetherian.
Indeed, let $G$ be an infinite group.  Then $U=\mathcal{U}(G, \text{id})$ is weakly right noetherian.  It has three $\mathcal{H}=\mathcal{J}$-classes: $G,$ $J=\{x_g : g\in G\}$ and $\{0\}.$  The principal factor of $J$ is isomorphic to $N_G,$ which is not weakly right noetherian.
\end{remark}

In the case that a semigroup $S_{\beta}, \beta\in Y,$ has local right identities, it is a necessary condition for $S=\mathcal{S}(Y, S_{\alpha})$ to be weakly right noetherian that $S_{\beta}$ be weakly right noetherian.

\begin{lemma}
\label{sofs, lri}
Let $S=\mathcal{S}(Y, S_{\alpha})$ be a semilattice of semigroups, let $\beta\in Y,$ and suppose that $S_{\beta}$ has local right identities. 
If $S$ is weakly right noetherian, then so is $S_{\beta}.$
\end{lemma}

\begin{proof}
Let $Y^{\prime}=\{\alpha\in Y : \alpha\ngeq\beta\},$ and let $I=\bigcup_{\alpha\in Y^{\prime}}S_{\alpha}.$
Now, $I$ is an ideal and $T=S\!\setminus\!I$ is a subsemigroup of $S,$ so $T$ is weakly right noetherian by Corollary \ref{complement left ideal}.
Since $S_{\beta}$ is an ideal of $T$ with local right identities, it is weakly right noetherian by Corollary \ref{right ideal, lri}.
\end{proof}

The following corollary follows from Lemmas \ref{sofs, lri} and \ref{finite union}.

\begin{corollary}
\label{sofs, lri corollary}
Let $S=\mathcal{S}(Y, S_{\alpha})$ be a semilattice of semigroups where $Y$ is finite and each $S_{\alpha}$ has local right identities.  Then $S$ is weakly right noetherian if and only if each $S_{\alpha}$ is weakly right noetherian.
\end{corollary}

We now consider the situation for strong semilattices of semigroups.

\begin{prop}
\label{strong semilattice}
Let $S=\mathcal{S}(Y, S_{\alpha}, \phi_{\alpha, \beta})$ be a strong semilattice of semigroups.  If $S$ is weakly right noetherian, then $Y$ is weakly noetherian and each $S_{\alpha}$ is weakly right noetherian.
\end{prop}

\begin{proof}
The semilattice $Y$ is weakly noetherian by Lemma \ref{semilattice}.  Now let $\alpha\in Y.$  We prove that $S_{\alpha}$ preserves $\mathcal{R}$ in $S,$ and hence $S_{\alpha}$ is weakly right noetherian by Proposition \ref{R preorder}.  We write $\mathcal{R}=\mathcal{R}_S$ and $\mathcal{R}_{\alpha}=\mathcal{R}_{S_{\alpha}}.$  We just need to show that $\leq_{\mathcal{R}}\cap~(S_{\alpha}\times S_{\alpha})\subseteq~\leq_{\mathcal{R}_{\alpha}}$.
So, let $a, b\in S_{\alpha}$ and $a\leq_{\mathcal{R}}b.$  Then $a=bs$ for some $s\in S^1.$  If $s=1$ then $a=b,$ so assume that $s\in S.$  Then $s\in S_{\beta}$ for some $\beta\in Y.$  Since $a=bs\in S_{\alpha}S_{\beta}\subseteq S_{\alpha\beta},$ we conclude that $\alpha\beta=\alpha.$  It follows that $$a=(b\phi_{\alpha, \alpha})(s\phi_{\beta, \alpha})=b(s\phi_{\beta, \alpha})\in bS_{\alpha},$$
so $a\leq_{\mathcal{R}_{\alpha}}b,$ as required.
\end{proof}

Proposition \ref{strong semilattice} and Lemma \ref{finite union} together yield:

\begin{corollary}
Let $S=\mathcal{S}(Y, S_{\alpha}, \phi_{\alpha, \beta})$ be a strong semilattice of semigroups where $Y$ is finite.  Then $S$ is weakly right noetherian if and only if each $S_{\alpha}$ is weakly right noetherian.
\end{corollary}

Example \ref{cr strong ex} below shows that the converse of Proposition \ref{strong semilattice} does not hold, even in the case that each $S_{\alpha}$ is finite.  

\begin{prob}
Find necessary and sufficient conditions for a strong semilattice of semigroups to be weakly right noetherian.
\end{prob}

\subsection{Rees matrix semigroups and Brandt extensions\nopunct}
\ \par
\vspace{0.5em}

Let $S$ be a semigroup, let $I$ and $J$ be two non-empty index sets, and let $P=(p_{ji})$ be a $J\times I$ matrix with entries from $S.$  The set $I\times S\times J$ becomes a semigroup under the multiplication given by 
$$(i, s, j)(k, t,l)=(i, sp_{jk}t, l),$$
and is called the {\em Rees matrix semigroup over $S$ with respect to $P$}.  We denote this semigroup by $\mathcal{M}(S; I, J; P).$\par
We now modify the Rees matrix construction as follows. 
Let the matrix $P$ have entries from $S^0.$
The set $(I\times S\times J)\cup\{0\}$ with multiplication given by $$(i, s, j)(k, t, l)=
\begin{cases}
 (i, sp_{jk}t, l)&\text{if }p_{jk}\in S\!\setminus\!\{0\}\\
 0&\text{if }p_{jk}=0,
\end{cases}$$
and $0(i, s, j)=(i, s, j)0=0^2=0,$ is a semigroup.  It is called the {\em Rees matrix semigroup with zero over $S$ with respect to $P$}, and is denoted by $\mathcal{M}^0(G; I, J; P).$\par
A semigroup is said to be {\em completely simple} if it is simple and contains minimal left and right ideals.  A semigroup with zero is said to be {\em completely $0$-simple} if it is $0$-simple and contains $0$-minimal left and right ideals.  Rees \cite{Rees} proved that a semigroup is completely $0$-simple if and only if it is isomorphic to a Rees matrix semigroup with zero $\mathcal{M}^0(G; I, J; P)$ over a group $G$ such that every row and column of $P$ contains at least one element in $G.$  Consequently, a semigroup is completely simple if and only if it is isomorphic to a Rees matrix semigroup $\mathcal{M}(G; I, J; P)$ over a group $G.$\par
We shall consider conditions under which a Rees matrix semigroup with zero $T=\mathcal{M}^0(S; I, J; P)$ is weakly right noetherian.  We begin by considering what affect $T$ being weakly right noetherian has on $S$ and the index sets $I$ and $J.$

\begin{lemma}
\label{Rees components}
Let $T=\mathcal{M}^0(S; I, J; P)$ be weakly right noetherian.
\begin{enumerate}[leftmargin=*]
\item The semigroup $S$ is weakly right noetherian and $I$ is finite.
\item Let $U$ be the ideal of $S^0$ generated by the entries of $P.$  If the set $S\!\setminus\!U$ is non-empty, then both $S\!\setminus\!U$ and $J$ are finite.
\end{enumerate}
\end{lemma}

\begin{proof}
(1) Let $A$ be a right ideal of $S.$  Then $B=(I\times A\times J)\cup\{0\}$ is a right ideal of $T.$  Since $T$ is weakly right noetherian, there exists a finite set $U\subseteq B$ such that $B=UB^1.$  We may assume that $U=I_0\times X\times J_0$ for some finite sets $I_0\subseteq I, X\subseteq A, J_0\subseteq J.$  We claim that $I=I_0$ and $A=XS^1.$  Indeed, let $i\in I,$ $a\in A,$ and pick any $j\in J .$ Then $(i, a , j)=(i_0, x, j_0)t$ for some $(i_0, x, j_0)\in U$ and $t\in T^1.$  It follows that $i=i_0\in I_0$ and $a\in xS^1\subseteq XS^1,$ as required.\par
(2) Notice that the set $I\times (S\!\setminus\!U)\times J$ consists of indecomposable elements of $T$; it is hence finite by Lemma \ref{indecomposable}.  In particular, both $S\!\setminus\!U$ and $J$ are finite. 
\end{proof}

The converse of Lemma \ref{Rees components}(1) does not hold in general.  In order to show this, we first present the following lemma.

\begin{lemma}
\label{Rees subsemigroup}
Let $T=\mathcal{M}^0(S; I, J; P),$ let $J^{\prime}$ be a subset of $J,$ let $P^{\prime}=(p_{ji})_{j\in J^{\prime}, i\in I}$, and let $T^{\prime}=\mathcal{M}^0(S; I, J^{\prime}; P^{\prime}).$
If $T$ is weakly right noetherian, then so is $T^{\prime}.$ 
\end{lemma}

\begin{proof}
It is easy to see that $T^{\prime}$ is a right unitary subsemigroup of $T,$ so it is weakly right noetherian by Corollary \ref{right unitary}.
\end{proof}

\begin{ex}
\label{Rees ex}
Let $T=\mathcal{M}^0(S; I, J; P)$ with $S$ infinite, and suppose there exists $j\in J$ such that $p_{ji}=0$ for all $i\in I.$  Then $T^{\prime}=\mathcal{M}^0(S; I, \{j\}; P^{\prime}),$ where $P^{\prime}=(p_{ji})_{i\in I},$ is not weakly right noetherian by Lemma \ref{Rees components}(2) (in fact, $T^{\prime}$ is an infinite null semigroup).  It follows from Lemma \ref{Rees subsemigroup} that $T$ is not weakly right noetherian.
\end{ex}

\begin{remark}
Let $T=\mathcal{M}(S; I, J; P)$ be a Rees matrix semigroup over an infinite semigroup $S$ with a zero $0$ adjoined, where $0\notin S.$  If there exists $j\in J$ such that $p_{ji}=0$ for all $i\in I,$ then $T^{\prime}=\mathcal{M}^0(S; I, J; P)$ is not weakly right noetherian by Example \ref{Rees ex}.  Since $T^{\prime}$ is isomorphic to the Rees quotient of $T$ by the ideal $I\times\{0\}\times J,$ we deduce from Lemma \ref{quotient} that $T$ is not weakly right noetherian.
\end{remark}

The converse of Lemma \ref{Rees components}(1) holds in the case that every row of the matrix contains a unit.

\begin{prop}
\label{Rees, monoid}
Let $S=\mathcal{M}^0(M; I, J; P)$ be a Rees matrix semigroup over a monoid $M,$ and suppose that for every $j\in J$ there exists $i\in I$ such that $p_{ji}\in U(M).$
Then $S$ is weakly right noetherian if and only if $M$ is weakly right noetherian and $I$ is finite.
\end{prop}

\begin{proof}
The direct implication follows from Lemma \ref{Rees components}.\par
For the converse, let $A$ be a right ideal of $S.$  
Note that if $(i, u, j)\in A,$ then $(i, u, l)\in A$ for all $l\in J.$  Indeed, there exists $k\in I$ such that $p_{jk}\in U(M),$ so $(i, u, l)=(i, u, j)(k, p_{jk}^{-1}, l)\in A.$
Now fix $j_0\in J,$ and choose $i_0\in I$ such that $p_{j_0i_0}\in U(M).$  Let $I^{\prime}$ be the set of elements of $I$ that appear in $A.$  For each $i\in I^{\prime},$ define a set $$A_i=\{u\in M : (i, u, j_0)\in A\}.$$
We claim that $A_i$ is a right ideal of $S.$  Indeed, if $u\in A_i$ and $m\in M,$ then 
$$(i, um, j_0)=(i, u, j_0)(i_0, p_{j_0i_0}^{-1}m, j_0)\in A_i.$$
Since $M$ is weakly right noetherian, there exist finite sets $X_i\subseteq A_i, i\in I^{\prime},$ such that $A_i=X_iM.$  We claim that $A$ is generated by the finite set
$$Y=\{(i, x, j_0) : i\in I^{\prime}, x\in X_i\}.$$
Indeed, if $(i, u, j)\in A$ then $(i, u, j_0)\in A,$ so $u=xm$ for some $x\in X_i$ and $m\in M.$  Hence, we have $(i, u, j)=(i, x, j_0)(i_0, p_{j_0i_0}^{-1}m, j)\in YM,$ as required.
\end{proof}

\begin{corollary}
\label{completely simple}
Let $S$ be a completely (0-)simple semigroup.  Then $S$ is weakly right noetherian if and only if it has finitely many $\mathcal{R}$-classes.
\end{corollary}

\begin{proof}
The semigroup $S^0$ is completely $0$-simple.  It follows from Proposition \ref{Rees, monoid} and Corollary \ref{large} that $S$ is weakly right noetherian if and only if it has finitely many $\mathcal{R}$-classes.
\end{proof}

We now consider Brandt extensions.  Let $S$ be a semigroup and let $I$ be a non-empty set.  The set $(I\times S\times I)\cup\{0\}$ becomes a semigroup under the multiplication given by $$(i, s, j)(k, t,l)=
\begin{cases}
(i, st, l)&\text{if }j=k\\
0&\text{ otherwise,}
\end{cases}$$ 
and $0x=x0=0$ for all $x\in(I\times S\times I)\cup\{0\}.$
It is called the {\em Brandt extension} of $S$ by $I,$ and we denote it by $\mathcal{B}(S, I).$\par
Notice that if $S$ is a monoid, then $\mathcal{B}(S, I)$ is isomorphic to $\mathcal{M}^0(S; I, I; P)$ where $P$ is the $I\times I$ identity matrix.  Brandt extensions of groups are precisely the completely $0$-simple inverse semigroups \cite[Theorem V.5.1]{Petrich}.

\begin{prop}
Let $S$ be a semigroup and let $I$ be a non-empty set.  Then the Brandt extension $\mathcal{B}(S, I)$ is weakly right noetherian if and only if $S$ is weakly right noetherian and $I$ is finite.
\end{prop}

\begin{proof}
Let $T=\mathcal{B}(S, I).$  We just need to prove that if $T$ is weakly right noetherian, then $I$ is finite.  Indeed, if $I$ is finite, then $\mathcal{B}(S^1, I)\!\setminus\!T$ is finite.  It then follows from Corollary \ref{large} and Proposition \ref{Rees, monoid} that $T$ is weakly right noetherian if and only if $S$ is weakly right noetherian.\par 
So, suppose that $T$ is weakly right noetherian. Then there exists a finite set $U\subseteq T$ such that $T=UT^1.$  Let $I_0$ be the elements of $I$ appearing in $U.$  Let $i\in I,$ and pick any $s\in S.$  Then $(i, s, i)=(i_1, x, i_2)t$ for some $(i_1, x, i_2)\in U$ and $t\in T^1.$  It follows that $i=i_1=i_2\in I_0,$ and hence $I=I_0$ is finite.
\end{proof}

\subsection{Bruck-Reilly extensions\nopunct}
\ \par
\vspace{0.5em}
Let $M$ be a monoid with identity $1_M$ and let $\theta : M\to M$ be an endomorphism.
We define a binary operation on the set $\mathbb{N}_0\times M\times\mathbb{N}_0$ by
$$(j, a, k)(p, b, q)=(j-k+t, (a\theta^{t-k})(b\theta^{t-p}), q-p+t),$$
where $t=\max(k, p)$ and $\theta^0$ denotes the identity map on $M.$
With this operation the set $\mathbb{N}_0\times M\times\mathbb{N}_0$ is a monoid with identity $(0, 1_M, 0).$  It is denoted by $BR(M, \theta)$ and is called the {\em Bruck-Reilly extension of $M$ determined by $\theta$}.\par
The {\em bicyclic monoid} is the set $\mathbb{N}_0\times\mathbb{N}_0$ with multiplication given by $$(j, k)(p, q)=(j-k+t, q-p+t),$$ where $t=\max(k, p).$  
Clearly the bicyclic monoid is a homomorphic image of $BR(M, \theta).$
We note that the bicyclic monoid is a simple inverse monoid and is defined by the presentation 
$\langle b, c\,|\,bc=1\rangle.$
It is well known that every one-sided ideal of the bicyclic monoid is principal, so we certainly have:

\begin{lemma}
\label{bicyclic}
The bicyclic monoid is weakly noetherian.
\end{lemma}

We shall provide necessary and sufficient conditions for a Bruck-Reilly extension to be weakly right noetherian.  In order to do so, we first make the following definition.\par
Let $M$ be a monoid and let $\theta : M\to M$ be an endomorphism.
We call a sequence $(I_j)_{j\in\mathbb{N}_0}$ of right ideals of $M$ a {\em $\theta$-sequence} if $I_j\theta\subseteq I_{j+1}$ for every $j\geq 0.$

\begin{thm}
\label{BR}
Let $M$ be a monoid and let $\theta : M\to M$ be an endomorphism.  Then $BR(M, \theta)$ is weakly right noetherian if and only if the following conditions hold:
\begin{enumerate}
\item $M$ is weakly right noetherian;
\item for any $\theta$-sequence $(I_j)_{j\in\mathbb{N}_0}$ of right ideals of $M,$ there exists some $n\in\mathbb{N}_0$ with a finite set $Y\subseteq I_n$ such that $I_{n+r}=(Y\theta^r)M$ for all $r\geq 0.$
\end{enumerate}
\end{thm}

\begin{proof}
We denote $BR(M, \theta)$ by $N.$  Note that for any right ideal $I$ of $N,$ we have $(j, a, k)\in I$ if and only if $(j, a, 0)\in I.$ \par
($\Rightarrow$)  The monoid $M$ is isomorphic to the submonoid $M_0=\{0\}\times M\times\{0\}$ of $N.$
It can easily be shown $M_0$ is right unitary in $N,$ so $M\cong M_0$ is weakly right noetherian by Corollary \ref{right unitary}.\par
Now let $(I_j)_{j\in\mathbb{N}_0}$ be a $\theta$-sequence of right ideals of $M.$
We define a set $$I=\bigcup_{j\in\mathbb{N}_0}\{(j, a, k) : a\in I_j, l\in\mathbb{N}_0\}.$$
We claim that $I$ is a right ideal of $N.$  Indeed, let $(j, a, k)\in I$ and $(p, m, q)\in N.$  Then $A\in I_j.$  Let $u$ denote the element
$$(j, a, k)(p, m, q)=(j-k+t, (a\theta^{t-k})(m\theta^{t-p}), q-p+t),$$
where $t=\max(k, p).$  If $t=k,$ then $u=(j, a(m\theta^{k-p}), q-p+k)\in I$ since $I_j$ is a right ideal of $M.$  If $t=p,$ then $u=(j-k+p, (a\theta^{p-k})m, q)\in I,$ since $(a\theta^{p-k})\in I_{j-k+p}$ and $I_{j-k+p}$ is a right ideal of $M.$\par
Since $N$ is weakly right noetherian, there exists a finite set $X\subseteq I$ such that $I=XN.$  By the note given at the beginning of the proof, we may assume that the third coordinate of each element of $X$ is $0.$
Set $n=\max\{j : (j, a, 0)\in X\},$ and let 
$$Y=\{a\theta^{n-j} : (j, a, 0)\in X\}\subseteq I_n.$$
Let $r\geq 0$ and let $b\in I_{n+r}.$  Then $(n+r, b, 0)\in I,$ so there exist $(j, a, 0)\in X$ and $(p, m, q)\in N$ such that $(n+r, b, 0)=(j, a, 0)(p, m, q).$
It follows that $n+r=j+p,$ $b=(a\theta^p)m$ and $q=0.$
Hence, we have that 
$$b=(a\theta^{n+r-j})m=\bigl((a\theta^{n-j})\theta^r\bigr)m\in(Y\theta^r)M,$$
so $I_{n+r}=(Y\theta^r)M,$ as required.\par
($\Leftarrow$)  Let $I$ be a right ideal of $N.$  For each $j\in\mathbb{N}_0$ define a set $$I_j=\{a\in M : (j, a, 0)\in I\}.$$
Clearly $I_j$ is either empty or a right ideal of $N.$
Let $n_0$ be minimal such that $I_{n_0}$ is non-empty.
For any $j\geq n_0,$ we have
$$a\in I_j\implies(j, a, 0)\in I\implies(j+1, a\theta, 0)=(j, a, 0)(1, 1_M, 0)\in I\implies a\theta\in I_{j+1},$$
so $I_j\theta\subseteq I_{j+1}.$  Thus $(I_j)_{j\geq n_0}$ is a $\theta$-sequence of right ideals of $M.$
By assumption, there exists $n\geq n_0$ with a finite set $Y\subseteq I_n$ such that $I_{n+r}=(Y\theta^r)M$ for all $r\geq 0.$
Since $M$ is weakly right noetherian, for each $j\in\{n_0, \dots, n-1\}$ there exists a finite set $X_j\subseteq I_j$ such that $I_j=X_jM.$  Writing $Y=X_n,$ we claim that $I$ is generated by the finite set
$$X=\bigcup_{j=n_0}^n\{(j, x, 0) : x\in X_j\}.$$
Let $(j, a, k)\in I.$  Then $a\in I_j$ and $j\geq n_0.$  If $j<n,$ then $a=xm$ for some $x\in X_j$ and $m\in M,$ so we have $$(j, a, k)=(j, x, 0)(0, m, k)\in XM.$$  
If $j\geq n,$ then $a=(x\theta^{j-n})m$ for some $x\in X_n$ and $m\in M,$ and hence $$(j, a, k)=(n, x, 0)(j-n, m, k)\in XM,$$
as required.
\end{proof}

\begin{corollary}
Let $M$ be a monoid and let $\theta : M\to U(M)$ be a homomorphism.  Then $BR(M, \theta)$ is weakly right noetherian if and only if $M$ is weakly right noetherian.
\end{corollary}

\begin{proof}
The forward direction follows immediately from Theorem \ref{BR}.  For the converse, let $(I_j)_{j\in\mathbb{N}_0}$ be a $\theta$-sequence of right ideals of $M.$  For any $j\in\mathbb{N}_0,$ we have $I_j\theta\in U(M)$ and $I_j\theta\subseteq I_{j+1},$ so $I_{j+1}\cap U(M)\neq\emptyset.$  Since $I_{j+1}$ is a right ideal of $M,$ it follows that $I_{j+1}=M.$  Thus $M=I_1=I_2=\cdots$, and $I_{1+r}=(\{1_M\}\theta^r)M$ for all $r\geq 0.$  Hence, $BR(M, \theta)$ is weakly right noetherian by Theorem \ref{BR}.
\end{proof}

We deduce from Theorem \ref{BR} that the Bruck-Reilly extension of a weakly right noetherian monoid (indeed, even a finite monoid) need not be weakly right noetherian.

\begin{lemma}
\label{BR lemma}
Let $M$ be a monoid.  Suppose there exists a right ideal $I$ of $M,$ an element $a\in M\!\setminus\!I,$ and a monoid homomorphism $\theta : M\to M$ such that $(I\cup\{a\})\theta\subseteq I.$
Then $BR(M, \theta)$ is not weakly right noetherian.
\end{lemma}

\begin{proof}
Let $J$ be the right ideal $aM\cup I,$ and consider the infinite $\theta$-sequence $J, J, \dots.$  
Since $J\theta\subseteq I\subsetneq J,$ this sequence does not satisfy the condition in (2) of Theorem \ref{BR}, so $BR(M, \theta)$ is not weakly right noetherian.
\end{proof}

\begin{remark}
The condition of Lemma \ref{BR lemma} is satisfied by the bicyclic monoid: let $a=(1, 0),$ let $I=(2, 0)B,$ and let $\theta : B\to B$ be given by $(i, j)\theta=(2i, 2j).$\par
This condition is also satisfied by any monoid $M$ such that $M\!\setminus\!U(M)$ is an ideal containing an idempotent $e$ and an element $a\notin eM$ (e.g.\ the 2-element null semigroup with an identity adjoined).  Indeed, let $I=eM$ and let $\theta : M\to M$ be the endomorphism given by $U(M)\theta=\{1\}$ and $\bigl(M\!\setminus\!U(M)\bigr)\theta=\{e\}.$  Then $(I\cup\{a\})\theta=\{e\}\subseteq I.$
\end{remark}

\section{\large{Regular Semigroups}\nopunct}
\label{sec:reg}

In this section we study weakly right noetherian regular semigroups.  We begin with a necessary and sufficient condition for a regular semigroup to be weakly right noetherian.  We then focus our attention on certain classes of regular semigroups, including inverse semigroups, completely regular semigroups and regular semigroups with a principal series.

\begin{thm}
\label{regular}
The following are equivalent for a regular semigroup $S$:
\begin{enumerate}
\item $S$ is weakly right noetherian;
\item for every subset $U\subseteq E(S),$ there exists a finite set $X\subseteq U$ with the following property: for each $u\in U$ there exists $x\in X$ such that $u=xu.$
\end{enumerate}
\end{thm}

\begin{proof}
$(1)\Rightarrow(2).$  Let $U$ be a subset of $E(S),$ and let $I$ be the right ideal $US^1$ of $S.$  Since $S$ is weakly right noetherian, there exist a finite subset $X\subseteq U$ such that $I=XS^1.$  For each $u\in U,$ we have that $u=xs$ for some $x\in X$ and $s\in S^1,$ and hence $u=x^2s=xu.$\par
$(2)\Leftarrow(1).$  Let $I$ be a right ideal of $S.$  By assumption, there exists a finite set $X\subseteq I\cap E(S)$ satisfying the property in (2).  We claim that $I=XS.$  Indeed, let $a\in I.$  Let $b$ be an inverse of $a,$ so $a=aba.$  Then $ab\in I\cap E(S),$ so there exists $x\in X$ such that $ab=x(ab)$.  Thus $a=aba=xa\in XS,$ as required.
\end{proof}

Let $S$ be a regular semigroup, and let $T=\langle E(S)\rangle$ be the subsemigroup of $S$ generated by its set of idempotents.  Then $T$ is regular by \cite[Corollary 2]{Fitz-Gerald}.  Since $E(S)=E(T),$ we immediate deduce from Theorem \ref{regular}:

\begin{corollary}
A regular semigroup $S$ is weakly right noetherian if and only if its subsemigroup $T=\langle E(S)\rangle$ is weakly right noetherian.
\end{corollary}

\begin{corollary}
\label{inverse}
An inverse semigroup $S$ is weakly right noetherian if and only if its semilattice of idempotents $E(S)$ is weakly noetherian.
\end{corollary}

\begin{corollary}
\label{inverse left-right}
An inverse semigroup is weakly right noetherian if and only if it is weakly noetherian.
\end{corollary}

\begin{remark}
Corollary \ref{inverse left-right} does not hold for regular semigoups in general.  
Indeed, any infinite right zero semigroup is weakly right noetherian but not weakly left noetherian.
\end{remark}

Let $X$ be an infinite set, and let $X^{-1}=\{x^{-1} : x\in X\}$ be a set disjoint from $X.$
The {\em polycyclic monoid} over $X,$ denoted by $P_X,$ is the monoid with zero defined by the presentation 
$$\langle X, X^{-1}\,|\,xx^{-1}=1, xy^{-1}=0\,(x, y\in X, x\neq y)\rangle.$$
This presentation yields the normal form $\{u^{-1}v : u, v\in X^*\}\cup\{0\}$ for $P_X$ \cite[Section 1]{Meakin}.
The monoid $P_X$ is an inverse monoid with a single non-zero $\mathcal{D}$-class (see \cite[Section 1.3]{Perrot} or \cite[p.\ 478]{Campbell}).  In the case that $|X|=1,$ $P_X$ is the bicyclic monoid with zero adjoined.  It turns out that this is the only case where $P_X$ is weakly noetherian.

\begin{prop}
The polycyclic monoid $P_X$ is weakly noetherian if and only if $|X|=1.$
\end{prop}

\begin{proof}
If $|X|=1,$ then $P_X$ is weakly noetherian by Lemma \ref{bicyclic} and Corollaries \ref{large} and \ref{inverse left-right}.\par 
Suppose $|X|\geq 2.$  Choose distinct elements $x, y\in X.$  We claim that the infinite set $\{y^{-i}x^{-1}xy^i : i\in\mathbb{N}\}\subseteq E(P_X)$ is an antichain.
Indeed, since $xy^{-1}=yx^{-1}=0,$ we deduce that for any $i\neq j,$ 
$$(y^{-i}x^{-1}xy^i )(y^{-j}x^{-1}xy^j)=y^{-i}x^{-1}(xy^{i-j}x^{-1})xy^j=0.$$
It now follows from Proposition \ref{wn semilattice} that $E(P_X)$ is not weakly noetherian, and hence $P_X$ is not weakly noetherian by Corollary \ref{inverse}.
\end{proof}

We now show that the principal factors of a regular semigroup inherit the property of being weakly right noetherian.

\begin{lemma}
\label{reg, principal factor}
Let $S$ be a regular semigroup.  If $S$ is weakly right noetherian, then so are all its principal factors.
\end{lemma}

\begin{proof}
Let $J$ be a $\mathcal{J}$-class of $S,$ fix $x\in J,$ and let $T=S^1xS^1.$  Now, $T$ is a union of $\mathcal{J}$-classes of $S,$ of which $J$ is the unique maximal one.  Since any pair of elements of $S$ that are inverses of each other must belong to the same $\mathcal{J}$-class, it follows that $T$ is regular, and hence $T$ is weakly right noetherian by Corollary \ref{regular subsemigroup}.
Then the principal factor of $J$ is weakly right noetherian by Lemma \ref{quotient}, since it is a Rees quotient of $T.$
\end{proof}

Lemma \ref{reg, principal factor} and Corollary \ref{principal series} together yield:

\begin{corollary}
\label{reg, principal series}
Let $S$ be a regular semigroup with a principal series.
Then $S$ is weakly right noetherian if and only if all its principal factors are weakly right noetherian.
\end{corollary}

A semigroup $S$ is said to be {\em completely semisimple} if all its principal factors are completely $0$-simple or completely simple.

\begin{corollary}
\label{completely semisimple}
Let $S$ be a completely semisimple semigroup with a principal series.
Then $S$ is weakly right noetherian if and only if it has finitely many $\mathcal{R}$-classes.
\end{corollary}

\begin{proof}
If $S$ is weakly right noetherian, then every principal factor of $S$ has finitely many $\mathcal{R}$-classes by Lemma \ref{reg, principal series} and Corollary \ref{completely simple}.  It follows that every $\mathcal{J}$-class is a finite union of $\mathcal{R}$-classes.  Since $S$ has finitely many $\mathcal{J}$-classes, we conclude that it has finitely many $\mathcal{R}$-classes.\par 
The converse follows from Corollary \ref{R-classes}.
\end{proof}

\begin{remark}
Corollary \ref{completely semisimple} does not hold is we remove the condition that $S$ has a principal series.
Indeed, there exist infinite weakly noetherian semilattices (in which $\mathcal{J}=\mathcal{R}$ is the identity relation).\par
Also, Corollary \ref{completely semisimple} does not hold for general regular semigroups with a principal series.
For example, the bicyclic monoid has a single $J$-class and is weakly right noetherian, but it has infinitely many $\mathcal{R}$-classes.
\end{remark}

A semigroup is said to be {\em completely regular} if it is a union of groups.  Completely regular semigroups have the following characterisation.

\begin{thm}\cite[Theorem 4.1.3]{Howie}
\label{cr characterisation}
Every completely regular semigroup is a semilattice of completely simple semigroups.
\end{thm}

From Lemmas \ref{semilattice} and \ref{sofs, lri}, and Corollary \ref{completely simple}, we deduce:

\begin{prop}
\label{cr}
Let $S$ be a completely regular semigroup, and let $S=\mathcal{S}(Y, S_{\alpha})$ be its decomposition into a semilattice of completely simple semigroups.
If $S$ is weakly right noetherian, then $Y$ is weakly noetherian and each $S_{\alpha}$ has finitely many $\mathcal{R}$-classes.
\end{prop}

We shall see that the converse of Proposition \ref{cr} does not hold.\par  
In the remainder of this section we focus our attention on \textit{strong} semilattices of completely simple semigroups.  For more information about the structure of such semigroups, see \cite[Section IV.4]{Petrich}.\par
A {\em Clifford semigroup} is an inverse completely regular semigroup.  It follows from Theorem \ref{cr characterisation} that Clifford semigroups are precisely the semilattices of groups.  In fact, Clifford semigroups are {\em strong} semilattices of groups \cite[Theorem III.2.12]{Grillet}.  If $S=\mathcal{S}(Y, G_{\alpha})$ is a semilattice of groups, it is clear that $Y$ is isomorphic to $E(S),$ so Corollary \ref{inverse} yields:

\begin{corollary}
Let $S$ be a Clifford semigroup with decomposition $S=\mathcal{S}(Y, G_{\alpha})$ into a semilattice of groups.  Then $S$ is weakly right noetherian if and only if $Y$ is weakly noetherian.
\end{corollary}

In what follows we shall provide necessary and sufficient conditions for a general strong semilattice of completely simple semigroups to be weakly right noetherian.  We use the following folklore result, which we prove for completeness.

\begin{lemma}
\label{cr strong lemma}
Let $S=\mathcal{S}(Y, S_{\alpha}, \phi_{\alpha, \beta})$ be a strong semilattice of completely simple semigroups.  Then each of Green's relations is a congruence on $S.$  Furthermore, we have $S/\mathcal{J}=S/\mathcal{D}=Y$; $S/\mathcal{R}$ is a strong semilattice of left zero semigroups; $S/\mathcal{L}$ is a strong semilattice of right zero semigroups; and $S/\mathcal{H}$ is a strong semilattice of rectangular bands.
\end{lemma}

\begin{proof}
It is clear that $\mathcal{D}=\mathcal{J}$ and $S/\mathcal{J}=Y.$
We prove that $\mathcal{R}$ is a congruence.  A dual argument proves that $\mathcal{L}$ is a congruence, and hence $\mathcal{H}=\mathcal{R}\cap\mathcal{L}$ is congruence.\par
For each $\alpha\in Y,$ let $S_{\alpha}=\mathcal{M}(G_{\alpha}; I_{\alpha}, J_{\alpha}; P_{\alpha}).$
Recall that $\mathcal{R}$ is a left congruence on $S,$ so we just need to show that it is a right congruence.
For $\alpha\in Y,$ we write $\mathcal{R}_{\alpha}=\mathcal{R}_{S_{\alpha}}.$  Since $S_{\alpha}$ is regular, we have that $\mathcal{R}_{\alpha}=\mathcal{R}\cap(S_{\alpha}\times S_{\alpha})$ \cite[Proposition 2.4.2]{Howie}.
Note that $\mathcal{R}_{\alpha}$ is a congruence on $S_{\alpha}.$\par
Let $(a, b)\in\mathcal{R}$ and $c\in S.$  Since $Y$ is $\mathcal{R}$-trivial, we must have that $(a, b)\in\mathcal{R}_{\alpha}$ for some $\alpha\in Y.$
Now, the element $c$ belongs to some $S_{\beta}, \beta\in Y.$
Certainly $(a\phi_{\alpha, \alpha\beta}, b\phi_{\alpha, \alpha\beta})\in\mathcal{R}_{\alpha\beta}.$
Since $\mathcal{R}_{\alpha\beta}$ is a congruence, we have
$$(ac, bc)=((a\phi_{\alpha, \alpha\beta})(c\phi_{\beta, \alpha\beta}), (b\phi_{\alpha, \alpha\beta})(c\phi_{\beta, \alpha\beta}))\in\mathcal{R}_{\alpha\beta}\subseteq\mathcal{R},$$
as required.\par
It can be easily shown that $S/\mathcal{R}\cong\mathcal{S}(Y, I_{\alpha}, \psi_{\alpha, \beta}),$ where the $I_{\alpha}$ are considered as left zero semigroups and each $\psi_{\alpha, \beta} : I_{\alpha}\to I_{\beta}$ is defined as follows: $i_{\alpha}\psi_{\alpha, \beta}=k_{\beta}$ if for some (and hence all) $j_{\alpha}\in J_{\alpha},$ we have $$(i_{\alpha}, p_{j_{\alpha}i_{\alpha}}^{-1}, j_{\alpha})\phi_{\alpha, \beta}=(k_{\beta}, p_{l_{\beta}k_{\beta}}^{-1}, l_{\beta})$$ for some $l_{\beta}\in J_{\beta}.$  Dually, we have that $S/\mathcal{L}$ is a strong semilattice of right zero semigroups.  Finally, we have $S/\mathcal{H}\cong\mathcal{S}(Y, B_{\alpha}, \theta_{\alpha, \beta}),$ where each $B_{\alpha}=I_{\alpha}\times J_{\alpha}$ is a rectangular band and each $\theta_{\alpha, \beta} : B_{\alpha}\to B_{\beta}$ is defined as follows: $(i_{\alpha}, j_{\alpha})\theta_{\alpha, \beta}=(k_{\beta}, l_{\beta})$ if 
$$(i_{\alpha}, p_{j_{\alpha}i_{\alpha}}^{-1}, j_{\alpha})\phi_{\alpha, \beta}=(k_{\beta}, p_{l_{\beta}k_{\beta}}^{-1}, l_{\beta}).$$
This completes the proof.
\end{proof}

\begin{thm}
\label{cr strong}
Let $S$ be a strong semilattice of completely simple semigroups, and let $S/\mathcal{R}=\mathcal{S}(Y, S_{\alpha}, \phi_{\alpha, \beta})$ be the decomposition of $S/\mathcal{R}$ into a strong semilattice of left zero semigroups.  Then $S$ is weakly right noetherian if and only if the following conditions hold:
\begin{enumerate}
\item $Y$ is weakly noetherian;
\item each $S_{\alpha}$ is finite;
\item there exists a finite subsemilattice $Y_0$ of $Y$ with the following property: for each $\beta\in Y,$ there exists $\alpha\in Y_0$ such that $\alpha\geq\beta$ and $\phi_{\alpha, \beta}$ is surjective.
\end{enumerate}
\end{thm}

\begin{proof}
Green's relation $\mathcal{R}$ is a congruence on $S$ by Lemma \ref{cr strong lemma}.  Hence, by Lemma \ref{congruence in R}, $S$ is weakly right noetherian if and only if $S/\mathcal{R}$ is weakly right noetherian.  Therefore, it suffices to consider the case that $S=S/\mathcal{R}.$\par
($\Rightarrow$) (1) and (2) follow immediately from Proposition \ref{cr}.  For (3), we have that $S=XS$ for some finite subset $X\subseteq S.$  Let $Y_0$ be the semigroup generated by $\{\alpha\in Y : X\cap S_{\alpha}\neq\emptyset\}.$  Since $X$ is finite and $Y$ is {\em locally finite} (that is, every finitely generated subsemigroup is finite), we conclude that $Y_0$ is finite.\par
Now consider $\beta\in Y,$ and let $S_{\beta}=\{b_1, \dots, b_n\}.$  For each $i\in\{1, \dots, n\},$ there exists $x_i\in X$ such that $b_i\in x_iS,$ which implies that $b_i=x_ib_i.$  Let $x_i\in S_{\alpha_i}.$
Since $S_{\beta}$ is a left zero semigroup, we have $b_i=(x_i\phi_{\alpha_i, \beta})b_i=x_i\phi_{\alpha_i, \beta}.$
Now set $\alpha=\alpha_1\dots\alpha_n.$  Then $\alpha\in Y_0,$ and since $\alpha_i\geq\beta$ for all $i\in\{1, \dots, n\},$ we deduce that $\alpha\geq\beta.$
For each $i\in\{1, \dots, n\},$ we have $b_i=(x_i\phi_{\alpha_i, \alpha})\phi_{\alpha, \beta},$ so $\phi_{\alpha, \beta}$ is surjective.\par
($\Leftarrow$) Let $I$ be a right ideal of $S.$  For each $\alpha\in Y_0$ and $a\in S_{\alpha},$ define a set
$$U_a=\{\beta\in Y : \beta\leq\alpha, a\phi_{\alpha, \beta}\in I\},$$
and let $I_a$ be the right ideal $U_aY$ of $Y.$  Since $Y$ is weakly noetherian, there exists a finite set $X_a\subseteq U_a$ such that $I_a=X_aY.$
We claim that $I$ is generated by the finite set
 $$X=\{a\phi_{\alpha, \beta} : a\in S_{\alpha}, \alpha\in Y_0, \beta\in X_a\}.$$
Indeed, let $c\in I\cap S_{\gamma}.$  There exists $\alpha\in Y_0$ such that $\alpha\geq\gamma$ and $\phi_{\alpha, \gamma}$ is surjective.  In particular, there exists $a\in S_{\alpha}$ such that $a\phi_{\alpha, \gamma}=c.$  Then $\gamma\in I_a,$ so $\gamma=\beta\gamma$ for some $\beta\in X_a.$  It follows that
$$c=c^2=(a\phi_{\alpha, \gamma})c=\bigl((a\phi_{\alpha, \beta})\phi_{\beta, \gamma}\bigr)(c\phi_{\gamma, \gamma})=(a\phi_{\alpha, \beta})c\in XS,$$
completing the proof of this direction and of the theorem.
\end{proof}

\begin{corollary}
\label{bound}
Let $S=\mathcal{S}(Y, S_{\alpha}, \phi_{\alpha, \beta})$ be a strong semilattice of completely simple semigroups.  If $S$ is weakly right noetherian, then the set $\{|S_{\alpha}/\mathcal{R}| : \alpha\in Y\}$ is bounded.
\end{corollary}

\begin{proof}
As in the proof of Lemma \ref{cr strong lemma}, we have that $S/\mathcal{R}$ is a a strong semilattice of semigroups $\mathcal{S}(Y, T_{\alpha}, \psi_{\alpha, \beta}),$ where $T_{\alpha}=S_{\alpha}/\mathcal{R}.$
By Lemma \ref{cr}, each $T_{\alpha}$ is finite.
Let $Y_0$ be as stated in Theorem \ref{cr strong}.  We claim that $\{|T_{\alpha}| : \alpha\in Y\}$ is bounded above by $$N=\max\{|T_{\alpha}|\in\mathbb{N} : \alpha\in Y_0\}.$$  Indeed, for each $\beta\in Y$ there exists $\alpha\in Y_0$ such that $\alpha\geq\beta$ and $\psi_{\alpha, \beta} : T_{\alpha}\to T_{\beta}$ is surjective; thus $|T_{\beta}|\leq|T_{\alpha}|\leq N.$ 
\end{proof}

Given Theorem \ref{cr strong}, we can show that the converse of Proposition \ref{cr} does not hold, even in the case that $S$ is a band.

\begin{ex}
\label{cr strong ex}
Let $S_i=\{x_i, y_i\}$ ($i\in\mathbb{N}$) be disjoint copies of the 2-element left zero semigroup.
Let $\phi_{i, i}$ be the identity map on $S_i,$ and for $i<j$ let $\phi_{i, j} : S_i\to S_j$ be the homomorphism given by $x_i\phi_{i, j}=y_i\phi_{i, j}=x_j.$
Then we have a strong semilattice of semigroups $S=\mathcal{S}(Y, S_i, \phi_{i, j}),$ where $Y$ is the infinite descending chain $(\mathbb{N}, \geq).$  Clearly $S$ does not satisfy condition (3) of Theorem \ref{cr strong}, so it is not weakly right noetherian.
\end{ex}

We end this section with an example demonstrating that Corollary \ref{bound} does not hold for completely regular semigroups in general.

\begin{ex}
Let $S_i=\{x_{i,1}, \dots, x_{i,i}\}$ for each $i\in\mathbb{N},$ and let $S=\bigcup_{i\in\mathbb{N}}S_i.$
For $i, j\in\mathbb{N},$ $k\in\{1, \dots, i\}$ and $l\in\{1, \dots, j\},$ define 
$$x_{i,k}x_{j,l}=\begin{cases}
x_{j,l}&\text{ if }i<j,\\
x_{i,k}&\text{ if }i\geq j.
\end{cases}$$
It can be shown that this multiplication is associative by an exhaustive case analysis.
It is easy to see that $S$ is a semilattice of semigroups $\mathcal{S}(Y, S_i),$
where $Y=(\mathbb{N}, \geq)$ and each $S_i$ is a left zero semigroup.\par
We now prove that $S$ is weakly right noetherian.  Let $I$ be a right ideal of $S,$ and let $i$ be minimal such that $I\cap S_i\neq\emptyset.$  We claim that $I=(I\cap S_i)S.$
Indeed, if $x_{j,l}\in I,$ then $j\geq i.$  If $i=j,$ then $x_{j,l}\in I\cap S_i.$
Otherwise, we have $x_{j,l}=x_{i,1}x_{j,l}\in(I\cap S_i)S,$ as required.
\end{ex}

\section{\large{Commutative Semigroups}\nopunct}
\label{sec:comm}

In this section we consider weakly noetherian commutative semigroups.
We begin by presenting the basic structure theory of commutative semigroups in terms of archimedean semigroups.\par
An {\em archimedean semigroup} is a commutative semigroup $S$ with the following property: for each $a, b\in S,$ there exist $n\in\mathbb{N}$ and $s\in S$ such that $a^n=bs.$  For instance, the free monogenic semigroup is archimedean.

\begin{thm}\cite[Theorem IV.2.2]{Grillet}
\label{archimedean decomposition}
Every commutative semigroup is a semilattice of archimedean semigroups.
\end{thm}

We now characterise archimedean semigroups with an idempotent.  We need the following definition.\par
A semigroup $S$ with zero $0$ is said to be {\em nilpotent} if for every $s\in S$ there exists $n\in\mathbb{N}$ such that $s^n=0.$ 

\begin{lemma}\cite[Proposition IV.2.3]{Grillet}
\label{Grillet, archimedean with idempotent}
A semigroup $S$ is archimedean with idempotent if and only if $S$ is either an abelian group or an ideal extension of an abelian group by a commutative nilpotent semigroup.
\end{lemma}

In general, archimedean semigroups can have a rather complex structure.
We refer the reader to \cite[Chapter IV]{Grillet} for more information.\par

In order for a commutative semigroup to be weakly noetherian, it is {\em not} necessary that all its archimedean components be weakly noetherian, as demonstrated by the following example.

\begin{ex}
\label{free comm ex}
Let $FC_2$ denote the free commutative semigroup on two generators $a$ and $b.$  We have that $FC_2$ is weakly noetherian by Theorem \ref{fg comm}.
It is easy to see that $FC_2$ has three archimedean components: $A=\langle a\rangle,$ $B=\langle b\rangle$ and $C=\{a^ib^j : i, j\geq 1\}.$  The infinite set $\{ab^i, a^ib : i\geq 1\}$ consists of all the indecomposable elements of $C,$ so $C$ is not weakly noetherian by Lemma \ref{indecomposable}.
\end{ex}

The next example shows that a commutative semigroup may not be weakly noetherian even if its structure semilattice and archimedean components are all weakly noetherian.

\begin{ex}
Let $S_i=\langle a_i\rangle$ ($i\in\mathbb{N}$) be disjoint copies of the free monogenic semigroup $\mathbb{N},$ which is weakly noetherian by Theorem \ref{fg comm}.
For $i<j,$ let $\phi_{i, j} : S_i\to S_j$ be the isomorphism given by $a_i\mapsto a_j.$
Let $S$ be the strong semilattice of archimedean semigroups $\mathcal{S}(Y, S_i, \phi_{i, j}),$ where $Y=(\mathbb{N}, \geq).$
Then $S$ contains an infinite set $\{a_i : i \in\mathbb{N}\}$ of indecomposable elements, and is hence not weakly noetherian by Lemma \ref{indecomposable}.
\end{ex}

We now state the main result of this section.

\begin{thm}
\label{comm, finitely many ACs}
Let $S$ be a commutative semigroup with finitely many archimedean components.  Then $S$ is weakly noetherian if and only if $S/\mathcal{H}$ is finitely generated.
\end{thm}

\begin{remark}
Note that if $S$ is a commutative semigroup such that $S/\mathcal{H}$ is finitely generated, then $S$ has finitely many archimedean components.  Indeed, let $S=\mathcal{S}(Y, S_{\alpha})$ be a decomposition of $S$ into a semilattice of archimedean semigroups.  It can be easily shown that if two elements of $S$ are $\mathcal{H}$-related, then they belong to the same archimedean component.  Thus $S/\mathcal{H}=\mathcal{S}(Y, T_{\alpha})$ for some semigroups $T_{\alpha}.$  Hence, $Y$ is a homomorphic image of $S/\mathcal{H}.$  Since $S/\mathcal{H}$ is finitely generated, we conclude that $Y$ is finite.\par
It follows that Theorem \ref{comm, finitely many ACs} does not hold if the condition that $S$ has finitely many archimedean components is dropped, since there certainly exist weakly noetherian commutative semigroups with infinitely many archimedean components; e.g.\ infinite weakly noetherian semilattices.
\end{remark}

In order to prove Theorem \ref{comm, finitely many ACs}, we first state and prove a few lemmas.

\begin{lemma}
\label{nilpotent}
Let $T$ be a commutative nilpotent semigroup.  If $T$ is finitely generated as a right ideal, then it is finite.
\end{lemma}

\begin{proof}
Since $T$ is finitely generated as a right ideal, there exists a finite set $X\subseteq T$ such that $T=XT^1.$
Let $U=\langle X\rangle.$ 
For each $x\in X,$ let $$m(x)=\text{min}\{n\in\mathbb{N} : x^n=0\},$$ and let $N=\prod_{x\in X}m(x).$  It can be easily shown that $|U|\leq N.$
We claim that $T=U.$  Suppose for a contradiction that $T\neq U,$ and let $a\in T\!\setminus\!U.$
We have that $a=x_1t_1$ for some $x_1\in X$ and $t_1\in T^1.$  Since $a\notin U,$ we have that $t_1\in T\!\setminus\!U.$  
By a similar argument, there exist $x_2\in X$ and $t_2\in T\!\setminus\!U$ such that $t_1=x_2t_2.$
Continuing in this way, for each $n\in\mathbb{N}$ there exist $x_1, \dots, x_n\in X$ and $t_n\in T\!\setminus\!U$ such that $a=(x_1\dots x_n)t_n.$
However, we have that $x_1\dots x_N=0$ and hence $a=0,$ which is a contradiction. 
\end{proof}

\begin{lemma}
\label{archimedean with idempotent}
Let $S$ be an archimedean semigroup with idempotent. Then $S$ is weakly noetherian if and only if $S$ is either a group or an ideal extension of a group by a finite nilpotent semigroup.
\end{lemma}

\begin{proof}
Every group is weakly noetherian, so assume that $S$ is not a group.
By Lemma \ref{Grillet, archimedean with idempotent}, there exists a group $G$ that is an ideal of $S$ such that $T=S/G$ is a nilpotent semigroup.
If $S$ is weakly noetherian, it follows from Lemmas \ref{quotient} and \ref{nilpotent} that $T$ is finite.
Conversely, if $T$ is finite, then $S$ is weakly noetherian by Corollary \ref{large}.
\end{proof}

\begin{lemma}
\label{archimedean without idempotent}
Let $S$ be an archimedean semigroup without idempotent. Then $S$ is weakly noetherian if and only if it is finitely generated.
\end{lemma}

\begin{proof}
The reverse implication follows from Theorem \ref{fg comm}, so we just need to prove the direct implication.\par
Let $a$ be a fixed element of $S.$  
The {\em Tamura order} on $S$ (with respect to $a$) is defined by $$x\leq_a y\iff x=a^ny\text{ for some }n\geq 0.$$
By \cite[Section IV.4]{Grillet}, there exists a set $M$ of maximal elements of $S$ (under $\leq_a$), where $a\in M,$
such that every element of $S$ can be written in the form $p_n=a^np$ with $n\geq 0$ and $p\in M,$ and the set $I=S\!\setminus\!M$ is an ideal.  
The Rees quotient $S/I$ is nilpotent, since $S$ is an archimedean semigroup, and hence it is finite by Lemmas \ref{quotient} and \ref{nilpotent}.
Therefore, the set $M$ is finite and $S=\langle M\rangle$ is finitely generated.
\end{proof}

\begin{lemma}
\label{AC fg}
Let $S$ be a weakly noetherian commutative semigroup with no non-trivial subgroups, and let $S=\mathcal{S}(Y, S_{\alpha})$ be a decomposition of $S$ into a semilattice of archimedean semigroups.
Let $\beta\in Y$ and let $T$ be the subsemigroup $\bigcup_{\alpha\geq\beta}S_{\alpha}$ of $S.$  Then there exists a finite set $X\subseteq S_{\beta}$ such that $S_{\beta}=\langle X\rangle(T\!\setminus\!S_{\beta})^1.$
\end{lemma}

\begin{proof}
Since the complement of $T$ is an ideal of $S,$ we have that $T$ is weakly noetherian by Corollary \ref{complement left ideal}.  Note that $S_{\beta}$ is an ideal of $T.$\par
Suppose first that $S_{\beta}$ has an idempotent.  Then $S_{\beta}$ is a nilpotent semigroup with zero $0.$ 
Since $T$ is weaky noetherian, we have that $S_{\beta}=XT^1$ for some finite set $X\subseteq S_{\beta}\!\setminus\!\{0\}.$  Then $S_{\beta}=(XS_{\beta}^1)(T\!\setminus\!S_{\beta})^1.$  By the same argument as the one in Lemma \ref{nilpotent}, we have that $XS_{\beta}^1$ is the finite semigroup $\langle X\rangle.$\par 
Now suppose that $S_{\beta}$ has no idempotent.  Let $M$ be the set of maximal elements of $S_{\beta}$ under the Tamura order with respect to an element $a\in S_{\beta},$ and let $I=S_{\beta}\!\setminus\!M.$  Then $I$ is an ideal of $T.$  We have that $T/I$ is weakly noetherian by Lemma \ref{quotient}, and $S_{\beta}/I$ is a nilpotent semigroup with zero $0.$  By the same argument as above, there exists a finite set $X\subseteq M=S_{\beta}\!\setminus\!\{0\}$ such that $S_{\beta}/I=\langle X\rangle(T\!\setminus\!S_{\beta})^1.$  We may assume without loss of generality that $a\in X.$  
Since every element of $S_{\beta}$ can be written as $a^np$ for some $n\geq 0$ and $p\in M,$ it follows that $S_{\beta}=\langle X\rangle(T\!\setminus\!S_{\beta})^1.$
\end{proof}

\begin{remark}
Given Lemma \ref{AC fg}, one might be tempted to think that in a weakly noetherian commutative semigroup with no non-trivial subgroups, every archimedean semigroup is contained in a finitely generated subsemigroup.  However, this is not the case.  Indeed, lettting $Y=(\mathbb{N}, \geq)$ and recalling Construction \ref{con}, we have that $U=\mathcal{U}(Y, \text{id})$ is weakly noetherian by Proposition \ref{con prop}.  Clearly $U$ is commutative.  It can be easily shown that $U$ is locally finite, so its infinite archimedean component $N_Y$ is not contained in a finitely generated subsemigroup.
\end{remark}

We are now in a position to prove the main result of this section.

\begin{proof}[Proof of Theorem \ref{comm, finitely many ACs}]
($\Rightarrow$) Let $S=\mathcal{S}(Y, S_{\alpha}),$ where $Y$ is finite, be the decomposition of $S$ into a semilattice of archimedean semigroups.
Let $T$ denote the quotient $S/\mathcal{H}.$  Then $T$ is a semilattice of archimedean semigroups $\mathcal{S}(Y, T_{\alpha})$ where each $T_{\alpha}$ is $\mathcal{H}$-trivial.  We have that $T$ is weakly noetherian by Lemma \ref{quotient}.\par
We prove that $T$ is finitely generated by induction on $|Y|.$  Suppose that $|Y|=1,$ so that $T$ is an archimedean semigroup.  If $T$ has an idempotent, then it follows from Lemma \ref{archimedean with idempotent} that $T$ is finite.  If $T$ has no idempotent, then $T\cong S$ is finitely generated by Lemma \ref{archimedean without idempotent}.\par 
Now suppose that $|Y|>1.$  Let $0$ be the minimal element of $Y.$  By Lemma \ref{AC fg}, there exists a finite set $X_0\subseteq T_0$ such that $T_0=\langle X_0\rangle(T\!\setminus\!T_0)^1.$  Let $\alpha_1, \dots, \alpha_k$ be the elements of $Y$ that only $0$ is strictly less than.  For each $i\in\{1, \dots, k\},$ define 
$$T_i=\bigcup_{\alpha\geq \alpha_i}T_{\alpha}.$$  
We have that $T_i$ is a subsemigroup of $T$ whose complement is an ideal, so it is weakly noetherian by Corollary \ref{complement left ideal}.  By the inductive hypothesis, $T_i$ is generated by some finite set $X_i.$  It follows that $T=\bigcup_{i=0}^kT_i$ is generated by the finite set $\bigcup_{i=0}^kX_i,$ as required.\par
($\Leftarrow$)  If $S/\mathcal{H}$ is finitely generated, then it is weakly noetherian by Theorem \ref{fg comm}, and hence $S$ is weakly noetherian by Lemma \ref{congruence in R}.
\end{proof}

From the proof of Theorem \ref{comm, finitely many ACs}, we deduce a couple of corollaries.  The first concerns {\em complete} semigroups; that is, commutative semigroups in which every archimedean component contains an idempotent.

\begin{corollary}
Let $S$ be a complete semigroup with finitely many idempotents.  Then $S$ is weakly noetherian if and only if $S/\mathcal{H}$ is finite.
\end{corollary}

\begin{proof}
We have that $S/\mathcal{H}=\mathcal{S}(Y, T_{\alpha}),$ where $Y$ is finite and each $T_\alpha$ is a nilpotent semigroup.  The direct implication is proved by a similar induction argument to the one in the proof of Theorem \ref{comm, finitely many ACs}.  Notice that $\langle X_0\rangle$ is finite, and also $T\!\setminus\!T_0$ is finite since it is a finite union of semigroups that are finite by the inductive hypothesis.  It follows that $T_0=\langle X_0\rangle(T\!\setminus\!T_0)^1$ is finite, and hence $T=T_0\cup(T\!\setminus\!T_0)$ is finite.
\end{proof}

On the other extreme we have:

\begin{corollary}
Let $S$ be an idempotent-free commutative semigroup with finitely many archimedean components.  Then $S$ is weakly noetherian if and only if it is finitely generated.
\end{corollary}

\begin{proof}
This proof is essentially the same as that of Theorem \ref{comm, finitely many ACs}, except we prove by induction that $S$ is finitely generated, rather than $S/\mathcal{H}.$
\end{proof}

We conclude this section by exhibiting an example of an idempotent-free commutative semigroup that is weakly noetherian but not finitely generated.

\begin{ex}
For each $i\in\mathbb{N},$ let $S_i=\langle a_i\rangle$ be a copy of the free monogenic semigroup $\mathbb{N},$ and let $S$ be the disjoint union of the semigroups $S_i.$  Define a multiplication on $S,$ extending those on each $S_i,$ as follows: for each $i, j, m, n\in\mathbb{N}$ with $i<j,$ let
$$a_i^ma_j^n=a_j^n=a_j^na_i^m.$$
It easy to see that, with this multiplication, $S$ is an idempotent-free commutative semigroup.
Moreover, we have that $S=\mathcal{S}(Y, S_i)$ where $Y=(\mathbb{N}, \geq).$  Since $Y$ is infinite, $S$ is not finitely generated.  It can be easily shown that every ideal of $S$ has the form $$a_i^mS^1=\{a_i^n : n\geq m\}\cup\biggl(\bigcup_{j>i}S_j\biggr)$$
for some $i, m\in\mathbb{N}.$  In particular, every ideal of $S$ is principal, so $S$ is weakly noetherian.
\end{ex}

\section*{Acknowledgements}
This research was supported by the London Mathematical Society through the Early Career Fellowship scheme.  The author would like to thank the referee for their helpful comments and suggestions, which improved the quality of the article.  The author is also grateful to Victoria Gould for her support, and in particular for comments that led to Theorem \ref{principal}, Proposition \ref {acc ideal} and Theorem \ref{BR}.


\begin{thebibliography}{99}
\bibitem{Aubert}
K. Aubert.  On the ideal theory of commutative semigroups.  \textit{Math. Scandinavica}, 1:39-54, 1953.  \textit{Monatsberichte Deutschen Akad. Wiss. Berlin}, 6:85-88, 1964.
\bibitem{Campbell}
C. Campbell, M. Quick, E. Robertson and G. Smith.  \textit{Groups St Andrews 2005: Volume 2}.  Cambridge University Press, 2007.
\bibitem{Clifford}
A. Clifford.  Semigroups containing minimal ideals.  \textit{Amer. J. Math.}, 70:521-526, 1948.
\bibitem{C&P}
A. Clifford and G. Preston.  \textit{The Algebraic Theory of Semigroups: Volume 2}.  American Math. Soc., 1967.
\bibitem{Dandan}
Y. Dandan, V. Gould, T. Quinn-Gregson and R. Zenab.  Semigroups with finitely generated universal left congruence.  \textit{Monat. Math.}, 190:689-724, 2019.
\bibitem{Davvaz}
Bijan Davvaz and Zahra Nazemian.  Chain conditions on commutative monoids.  \textit{Semigroup Forum}, 100:732-742, 2020.
\bibitem{Fitz-Gerald}
D. Fitz-Gerald.  On inverses of products of idempotents in regular semigroups.  \textit{J. Austral. Math. Soc.}, 13:335-337, 1972.
\bibitem{Goodearl}
K. Goodearl and R. Warfield.  \textit{An Introduction to Noncommutative Noetherian
Rings}.  Cambridge University Press, 2004.
\bibitem{Gould}
V. Gould, M. Hartmann and L. Shaheen.  On some finitary questions arising from the axiomatisability of certain classes of monoid acts.  \textit{Comm. Algebra}, 42:2584-2602, 2014.
\bibitem{Gray}
R. Gray and N. Ru{\v s}kuc.  Green index and finiteness condition for semigroups.  \textit{J.\ Algebra}, 320:3145-3164, 2008.
\bibitem{Grillet}
P. Grillet.  \textit{Semigroups: An Introduction to the Structure Theory.}  Marcel Dekker, Inc., 1995.
\bibitem{Hotzel}
E. Hotzel.  On semigroups with maximal conditions.  \textit{Semigroup Forum}, 11:337-362, 1975.
\bibitem{Howie}
J. Howie.  \textit{Fundamentals of Semigroup Theory}.  OUP Oxford, 1995.
\bibitem{Jespers}
E. Jespers and J. Okni{\'n}ski.  Noetherian semigroup algebras.  \textit{J. Algebra}, 218:543-564, 1999.
\bibitem{Keh}
N. Kehayopulu and M. Tsingelis.  Noetherian and Artinian ordered groupoids - semigroups.  \textit{Internat. J. Math. and Math. Sciences}, 2005:2041-2051, 2005.
\bibitem{Kilp}
M. Kilp, U. Knauer and A. Mikhalev.  \textit{Monoids, Acts and Categories}.  Walter de Gruyter, 2000.
\bibitem{Kozhukhov1}
I. Kozhukhov.  On semigroups with minimal or maximal condition on left congruences.  \textit{Semigroup Forum}, 21:337-350, 1980.
\bibitem{Kozhukhov2}
I. Kozhukhov.  Semigroups with certain conditions on congruences.  \textit{J. Math. Sciences}, 114:1119-1126, 2003.
\bibitem{Meakin}
J. Meakin and M. Sapir.  Congruences on free monoids and submonoids of polycyclic monoids.  \textit{J. Austral. Math. Soc.}, 54:236-254, 1993.
\bibitem{Miller1}
C. Miller and N. Ru{\v s}kuc.  Right noetherian semigroups.  \textit{Internat. J. Algebra Comput.}, 30:13-48, 2020.
\bibitem{Miller2}
C. Miller.  Semigroups for which every right congruence of finite index is finitely generated.  \textit{Monat. Math.}, 193:105-128, 2020.
\bibitem{Perrot}
J. Perrot.  Une famille de mono{\"i}des inversifs 0-bisimples g{\'e}n{\'e}ralisant le mono{\"i}de bicyclique.  \textit{S{\'e}minaire Dubreil.  Alg{\`e}bre et th{\'e}orie des nombres}, 25:1-15, 1971/72. 
\bibitem{Petrich}
M. Petrich.  \textit{Introduction to Semigroups}.  Merrill, 1973.
\bibitem{Redei}
L. R{\'e}dei.  Theorie der Endlich Erzeugbaren Kommutativen Halbgruppen.  \textit{B. G. Teubner Verlagsgesellshaft, Leipzig}, 1963.
\bibitem{Rees}
D. Rees.  On semi-groups.  \textit{Proc. Cambridge. Math. Soc.}, 3:387-400, 1940.
\bibitem{Sat}
M. Satyanarayana.  Semigroups with ascending chain condition.  \textit{London Math. Soc.}, 2-5:11-14, 1972.
\bibitem{Wallace}
A. Wallace.  Relative ideals in semigroups, II. The relations of Green.  \textit{Acta Math. Acad. Sci. Hung.}, 14:137-148, 1963.
\end{thebibliography}
\end{document}